\documentclass[12pt]{article}
\usepackage{graphicx}
\usepackage{amsmath}
\usepackage{amssymb}
\usepackage{theorem}

      \usepackage{hyperref}

\usepackage{fancyhdr}

\numberwithin{equation}{section}

\newtheorem{thm}{Theorem}[section]
\newtheorem{lemma}[thm]{Lemma}
\newtheorem{prop}[thm]{Proposition}
\newtheorem{cor}[thm]{Corollary}
{\theorembodyfont{\rmfamily}

\newtheorem{example}[thm]{Example}

\newtheorem{rmk}[thm]{Remark}
}

\newcommand{\qed}{\hfill \mbox{\raggedright \rule{.07in}{.1in}}}
 
\newenvironment{proof}{\vspace{1ex}\noindent{\bf
Proof}\hspace{0.5em}}{\hfill\qed\vspace{1ex}}
\newenvironment{pfof}[1]{\vspace{1ex}\noindent{\bf Proof of
#1}\hspace{0.5em}}{\hfill\qed\vspace{1ex}}

\newcommand{\R}{{\mathbb R}}
\newcommand{\Z}{{\mathbb Z}}

\newcommand{\C}{{\mathbb C}}
\newcommand{\D}{{\mathbb D}}
\newcommand{\T}{{\mathbb T}}

\newcommand{\barD}{{\overline{\D}}}

\newcommand{\cF}{{\mathcal F}}
\newcommand{\cH}{{\mathcal H}}
\newcommand{\cS}{{\mathcal S}}

\newcommand{\DE}{{\Delta_{\rm rapid}}}

\newcommand{\suma}{{\textstyle \sum_a}}
\newcommand{\Leb}{\operatorname{Leb}}
 \newcommand{\Int}{\operatorname{int}}
 \newcommand{\diam}{\operatorname{diam}}
 \newcommand{\dist}{\operatorname{dist}}

 \newcommand{\hX}{{\widehat X}}
 \newcommand{\hA}{{\widehat A}}
 \newcommand{\hB}{{\widehat B}}
 \newcommand{\hE}{{\widehat E}}
 \newcommand{\hJ}{{\widehat J}}
 \newcommand{\hR}{{\widehat R}}
 \newcommand{\hT}{{\widehat T}}
 \newcommand{\hV}{{\widehat V}}
 \newcommand{\hW}{{\widehat W}}
 \newcommand{\tw}{{\widetilde w}}
 \newcommand{\tv}{{\tilde v}}
 \newcommand{\tH}{{\widetilde H}}
 \newcommand{\tE}{{\widetilde E}}

 \newcommand{\bF}{{\bar F}}
 \newcommand{\bA}{{\bar A}}
 \newcommand{\bJ}{{\bar J}}
 \newcommand{\bV}{{\overline V}}
 \newcommand{\bW}{{\overline W}}
 \newcommand{\bY}{{\bar Y}}
 \newcommand{\bZ}{{\bar Z}}
 \newcommand{\bDelta}{{\bar \Delta}}

 \newcommand{\eps}{{\epsilon}}
 \newcommand{\bvarphi}{\bar\varphi}
 \newcommand{\bPhi}{\bar\Phi}
 \newcommand{\supp}{\operatorname{supp}}

\title{Sharp polynomial bounds on decay of correlations for multidimensional nonuniformly hyperbolic systems and billiards}

\author{
Henk Bruin\thanks{Faculty of Mathematics, University of Vienna, 1090 Vienna, AUSTRIA}
\and
Ian Melbourne\thanks{Mathematics Institute, University of Warwick, Coventry, CV4 7AL, UK}
\and
Dalia Terhesiu\thanks{Mathematisch Instituut, Niels Bohrweg 1, 2333 CA Leiden, THE NETHERLANDS}
}

\date{18 November 2018.  Revised 10 May 2020.}

\begin{document}

\maketitle

\begin{abstract}
Gou\"ezel and Sarig introduced operator renewal theory as a method to prove sharp results on polynomial decay of correlations for certain classes of nonuniformly expanding maps.  In this paper, we apply the method to 
planar dispersing billiards and multidimensional nonMarkovian intermittent maps.
\end{abstract}

\section{Introduction}

In two seminal papers, Young~\cite{Young98,Young99} obtained results on exponential and subexponential decay of correlations for nonuniformly hyperbolic dynamical systems.  In the case of subexponential decay, a natural question is to establish that the decay rates obtained in this way are optimal.  The first progress in this direction was by Sarig~\cite{Sarig02} who introduced the method of operator renewal theory.  This method was extended and refined by Gou\"ezel~\cite{Gouezel04a} and gives optimal results for one-dimensional intermittent maps of Pomeau-Manneville type~\cite{PomeauManneville80,Thaler80}.

A challenge has been to extend the applicability of operator renewal theory to higher-dimensional examples.  Two specific directions have required attention: (i) planar dispersing billiards, (ii) multidimensional nonMarkovian intermittent maps.  For results in these directions, we mention~\cite{HuVaientiapp,VaientiZhangsub}.

In this paper, we extend the operator renewal theory of Gou\"ezel and Sarig~\cite{Gouezel04a,Sarig02} to provide lower bounds in general situations where the Young tower method~\cite{Young99} provides upper bounds.  This includes directions~(i) and~(ii) above.  In the case of lower bounds for dispersing billiards, these are the first results using operator renewal theory, and the first results by any methods for billiards with decay rates other than $n^{-1}$.  For multidimensional intermittent maps, we obtain essentially optimal upper and lower bounds on decay of correlations.

Roughly speaking the result of Gou\"ezel and Sarig takes the following form.
Let $f:M\to M$ be an ergodic measure-preserving transformation defined on a probability space $(M,\mu)$.  The correlation function $\rho_{v,w}(n)$ is given by
\begin{equation} \label{eq:rho}
\textstyle \rho_{v,w}(n)=\int_M v\,w\circ f^n\,d\mu-\int_Mv\,d\mu\int_Mw\,d\mu
\end{equation}
for $L^2$ observables $v,\,w:M\to\R$.
For definiteness, as in~\cite{Gouezel04a,Sarig02} we consider one-dimensional Markovian intermittent maps such as in~\cite{LSV99} with $f(x)\approx x^{1+1/\beta}$ for $x$ near zero, where $\beta>1$,
and unique absolutely continuous invariant probability measure $\mu$.
Fix $\eta\in(0,1)$.
By~\cite{Young99}, there is a constant $C>0$ such that
\[
|\rho_{v,w}(n)|\le C{\|v\|}_{C^\eta}|w|_\infty\,n^{-(\beta-1)}
\]
for all $v\in C^\eta(M)$, $w\in L^\infty(M)$, $n\ge1$.
 Now fix a closed subset $X\subset M$ with $0\not\in X$
and let $h:X\to\Z^+$ denote the first return time to $X$.
By~\cite{Gouezel04a,Sarig02}, there exists a constant $C>0$ such that
\begin{equation} \label{eq:GS}
\textstyle \big|\rho_{v,w}(n)-\sum_{j>n}\mu(h>j)\int_M v\,d\mu\int_M w\,d\mu
\big|\le C{\|v\|}_{C^\eta}|w|_\infty\zeta_\beta(n)
\end{equation}
for all $v\in C^\eta(M)$ supported in $X$, $w\in L^\infty(X)$, $n\ge1$,
where 
\begin{equation} \label{eq:zeta}
\zeta_\beta(n)=\begin{cases}
n^{-\beta} & \beta>2 \\
n^{-2}\log n & \beta=2 \\
n^{-2(\beta-1)} & 1<\beta<2
\end{cases}.
\end{equation}
Since $\mu(h>n)\sim cn^{-\beta}$ for some $c>0$,
this shows that the results in~\cite{Young99} are sharp.
If in addition $\int v\,d\mu=0$, then $\rho_{v,w}(n)=O(n^{-\beta})$ for all $\beta>1$.  (One consequence of the main result in this paper is that the latter estimate holds for all $w\in L^\infty(M)$; this is not shown in previous papers.  See
Remark~\ref{rmk:main}.)

Abstract theorems for nonuniformly expanding and nonuniformly hyperbolic dynamical systems are stated in Sections~\ref{sec:main} and~\ref{sec:main2} respectively.
In common with the method of~\cite{Gouezel04a,Sarig02}, we induce on a convenient subset $Y\subset M$ with induced map $F:Y\to Y$ that is Gibbs-Markov for nonuniformly expanding maps and Gibbs-Markov after quotienting along local stable leaves for nonuniformly hyperbolic maps.  A key difference from~\cite{Gouezel04a,Sarig02} is that $F$ need not be a first return map.  As in~\cite{BruinTerhesiu18}, we are able to control the adverse effects associated with not being a first return and to obtain results that are essentially the same as those in~\cite{Gouezel04a,Sarig02}.

\begin{rmk}
We note that the setting in~\cite{VaientiZhangsub} is currently restricted to planar time-reversible systems.  
\end{rmk}

In the remainder of the introduction, we focus on the applications to
billiards and multidimensional intermittent systems.

\subsection{Billiard examples}
\label{sec:introbill}

Markarian~\cite{Markarian04} and Chernov \& Zhang~\cite{ChernovZhang05} considered a general framework for analysing decay of correlations for diffeomorphisms with singularities, with special emphasis on slowly mixing planar
dispersing billiards.
All known results on upper bounds for decay of correlations for dispersing billiards fall within this framework.  Within this framework, we obtain lower bounds.

The specific examples are described in more detail in Section~\ref{sec:billiard}.  
Here we summarize the results.
All integrals are with respect to Liouville measure.  Upper bounds are for general dynamically H\"older observables $v$ and $w$.  Lower bounds are for dynamically defined H\"older observables with nonzero mean supported in a suitable subset $X$ of phase space.

\vspace{1ex}
\noindent$\bullet$  {\bf Bunimovich stadia, semidispersing billiards, billiards with cusps.}
In these examples, the correlation decay rate $O(n^{-1})$ was established 
by~\cite{ChernovMarkarian07,ChernovZhang05,ChernovZhang08,Markarian04}.
By the argument in~\cite[Corollary~1.3]{BalintGouezel06}
(see also~\cite[Corollary~1.1]{BalintChernovDolgopyat11}), the result is essentially optimal in the sense that if $v=w$ and if $v$ is H\"older and satisfies a nondegeneracy condition, then $n\rho_{v,w}(n)\not\to0$ as $n\to\infty$.
However, for several years it remained an open question to obtain an asymptotic rate of the type~\eqref{eq:GS}.

We prove that for all three types of billiard there is a constant $c>0$ such that
$\rho_{v,w}(n)\sim cn^{-1}\int v\int w$.
The constant $c$ is given explicitly in terms of the billiard configuration space.
For example, in the case of a Bunimovich stadium with straight sides of length $\ell$, 
\[
c= \frac{4+3\log 3}{4-3\log 3}\,\frac{\ell^2}{4(\pi+\ell)}.
\]
(Throughout, $\log$ means logarithm to base $e$.)

A similar result for semidispersing billiards and billiards with cusps (but not stadia) can be found in~\cite{VaientiZhangsub}, though it is not clear that the asymptotic $\rho_{v,w}(n)\sim{\rm const.}\,n^{-1}$ is established there.

\vspace{1ex}
\noindent$\bullet$  {\bf Billiards with cusps at flat points.}
Correlation decay rates $O(n^{-(\beta-1)})$ with $\beta$ any prescribed value in $(1,2)$ were obtained in~\cite{Zhang17a}.
Here, $\beta$ corresponds to the flatness at the cusp.  We obtain the asymptotic $\rho_{v,w}(n)\sim cn^{-(\beta-1)}\int v\int w$. 
Again, the constant $c$ is given explicitly.

\vspace{1ex}
\noindent$\bullet$  {\bf Bunimovich flowers.}
The correlation decay rate $\rho_{v,w}(n)=O((\log n)^3n^{-2})$ was obtained in~\cite{ChernovZhang05}.
It is conjectured that the optimal rate is ${\rm const.}\,n^{-2}$.
We obtain the lower bound $\rho_{v,w}(n)\gg (\log n)^{-1}n^{-2}\int v\int w$.

\vspace{1ex}
\noindent$\bullet$  {\bf Dispersing billiards with vanishing curvature.}
Correlation decay rates $O((\log n)^\beta n^{-(\beta-1)})$ with $\beta$ any prescribed value in $(2,\infty)$ were obtained in~\cite{Zhang17a} for H\"older observables.
Here, $\beta$ corresponds to the flatness at the points of vanishing 
curvature.  
We obtain the lower bound $\rho_{v,w}(n)\gg (\log n)^{-1}n^{-(\beta-1)}\int v\int w$.

\subsection{Hu-Vaienti maps}
\label{sec:HuV}

We consider a class of
piecewise smooth multidimensional nonuniformly expanding intermittent maps $f:M\to M$, $M\subset\R^k$ compact, with a neutral fixed point.  
The case $k=1$ is very well-understood.  
Upper bounds on decay of correlations were obtained by~\cite{Hu04,Young99} and the results were shown to be sharp by~\cite{Gouezel04a,Hu04,Sarig02}.
Extending to multidimensional examples is relatively straightforward in the Markov case, but the nonMarkov case is very challenging because the standard symbolically H\"older spaces are unavailable for nonMarkov maps and there are difficulties using spaces of bounded variation in higher dimensions.
Also, as shown in~\cite{HuVaienti09}, such maps often have poor bounded distortion properties. 

Hu \& Vaienti~\cite{HuVaienti09} obtained results on existence of absolutely continuous ergodic invariant measures (both finite and infinite) for various classes of multidimensional nonMarkovian intermittent maps.  In a subsequent paper~\cite{HuVaientiapp}, first results on upper and lower bounds on decay of correlations were obtained.  As an application of the results in this paper, we obtain essentially optimal upper and lower bounds.

To fix ideas, we focus on~\cite[Example~5.1]{HuVaientiapp} as described in detail in Section~\ref{sec:HV}.  The neutral fixed point is taken to be at $0$ and 
$f(x)=x(1+|x|^\gamma+O(|x|^{\gamma'}))$ for $x$ close to $0$ where
$\gamma\in(0,k)$ and $\gamma'>\gamma$.
Using results of~\cite{AFLV11,HuVaienti09},
we show that  $\rho_{v,w}(n)=O(n^{-((k/\gamma)-1-\eps)})$ for $v$ H\"older and $w\in L^\infty$, where $\eps$ is arbitrarily small.
This is in marked contrast to~\cite{HuVaientiapp} who obtain results no better than $\rho_{v,w}(n)=O(n^{-((1/\gamma)-1)})$ in the multidimensional case $k\ge2$ and only for observables with support bounded away from $0$.

Moreover, our decay rate is essentially optimal.
For $v$ H\"older and $w\in L^\infty$ with supports bounded away from $0$ and nonzero mean,
we show that for any $\eps>0$
\begin{equation} \label{eq:HV}
n^{-((k/\gamma)-1+\eps)}\ll \rho_{v,w}(n)\ll n^{-((k/\gamma)-1-\eps)}.
\end{equation}

\vspace{1ex}
The remainder of the paper is organized as follows.  In Section~\ref{sec:prel}, we recall background material on inducing, Gibbs-Markov maps, Young towers, and Chernov-Markarian-Zhang structures.
Our main result for nonuniformly expanding maps is stated in Section~\ref{sec:main} and proved in Section~\ref{sec:proof}.
In Section~\ref{sec:tail}, we relate tail estimates for different return times.
In Section~\ref{sec:NUE}, we apply our results to multidimensional nonuniformly expanding maps including those mentioned in Subsection~\ref{sec:HuV}.

In Section~\ref{sec:main2}, we extend our main result to nonuniformly hyperbolic systems, including solenoidal versions of the maps in Section~\ref{sec:NUE}.
Finally, in Section~\ref{sec:billiard}, we consider the examples from billiards mentioned in Subsection~\ref{sec:introbill}.

\paragraph{Notation}
We use the ``big $O$'' and $\ll$ notation interchangeably, writing $a_n=O(b_n)$ or $a_n\ll b_n$ if there is a constant $C>0$ such that
$a_n\le Cb_n$ for all $n\ge1$.  Also, 
we write $a_n\approx b_n$ if $a_n\ll b_n\ll a_n$.
As usual, $a_n\sim b_n$ as $n\to\infty$ means that $\lim_{n\to\infty}a_n/b_n=1$.

Convolution of sequences $a_n$, $b_n$ ($n\ge0$) is denoted
$(a\star b)_n=\sum_{j=0}^n a_j b_{n-j}$.
Often we use the abuse of notation $a_n\star b_n$.
If $a_0$ is undefined (as for example $a_n=n^{-2}\log n$) then we redefine $a_0=1$ without mentioning it.  With these conventions we have the standard facts
$n^{-p}\star n^{-q}=O(n^{-q})$ and 
$n^{-p}\star n^{-q}\log n=O(n^{-q}\log n)$ for all $p>1$, $q\in(0,p]$.

\section{Preliminaries}
\label{sec:prel}

In this section, we recall background material on (one-sided)
Chernov-Markarian-Zhang structures.

\paragraph{Gibbs-Markov maps}

Let $(Y,\mu_Y)$ be a probability space with an at most countable measurable partition $\alpha$, and let $F:Y\to Y$ be an ergodic measure-preserving transformation.
For $\theta\in(0,1)$, define the {\em separation time} $s(y,y')$ to be the least integer $n\ge0$ such that $F^ny$ and $F^ny'$ lie in distinct partition elements in $\alpha$.
It is assumed that the partition $\alpha$ separates trajectories, so $s(y,y')=\infty$ if and only if $y=y'$;  then $\theta^s$ is a metric.

Let $\xi=\frac{d\mu_Y}{d\mu_Y\circ F}:Y\to\R$. 
We say that $F$ is a {\em (full-branch) Gibbs-Markov map} if 
\begin{itemize}
\item $F|_a:a\to Y$ is a measurable bijection for each $a\in\alpha$, and
\item 
There are constants $C>0$, $\theta\in(0,1)$ such that
$|\log\xi(y)-\log\xi(y')|\le C\theta^{s(y,y')}$ for all $y,y'\in a$, $a\in\alpha$.
\end{itemize}
A consequence is that there is a constant $C>0$ such that 
\begin{equation} \label{eq:GM}
\xi(y)\le C\mu_Y(a) \qquad\text{and} \qquad
|\xi(y)-\xi(y')|\le C\mu_Y(a)\theta^{s(y,y')},
\end{equation}
 for all 
$y,y'\in a$, $a\in\alpha$.

\paragraph{Return maps}
Suppose that $(M,\mu)$ is a probability space and that $f:M\to M$ is an ergodic measure-preserving transformation.  Fix a measurable subset $X\subset M$ with $\mu(X)>0$ and $h:X\to\Z^+$ integrable such that
$f^{h(x)}x\in X$ for all $x\in X$.  Then $h$ is called a {\em return time} and
$f^h:X\to X$ is called a {\em return map}.

If $h$ is the {\em first return time} to $X$ under $f$ (i.e.\ $h(x)=\inf\{n\ge1:f^nx\in X\}$), then $f^h:X\to X$ is called the {\em first return map} and
$\mu_X=(\mu|_X)/\mu(X)$ is an ergodic $f^h$-invariant probability measure on $X$.

\paragraph{Young towers}

Let $F:Y\to Y$ be a full-branch Gibbs-Markov map on $(Y,\mu_Y)$ with partition $\alpha$ and
let $\varphi:Y\to\Z^+$ be an integrable function constant on partition elements.  
We define the {\em (one-sided) Young tower} $\Delta=Y^\varphi$ and {\em tower map} $f_\Delta:\Delta\to\Delta$ as follows:
\[
\Delta=\{(y,\ell)\in Y\times\Z: 0\le \ell\le \varphi(y)-1\},
\quad
f_\Delta(y,\ell)=\begin{cases}
(y,\ell+1) & \ell\le\varphi(y)-2 \\
(Fy,0) & \ell=\varphi(y)-1
\end{cases}.
\]
Let $\bvarphi=\int_Y\varphi\,d\mu_Y$.
Then $\mu_\Delta=(\mu_Y\times{\rm counting})/\bvarphi$ is an ergodic $f_\Delta$-invariant probability measure on $\Delta$,
and it is mixing if and only if $\gcd\{\varphi(a):a\in\alpha\}=1$.
The tower has {\em exponential tails} if $\mu_Y(\varphi>n)=O(e^{-cn})$ for some $c>0$,
{\em polynomial tails} if $\mu_Y(\varphi>n)=O(n^{-\beta})$ for some $\beta>1$,
and {\em superpolynomial tails} if $\mu_Y(\varphi>n)=O(n^{-\beta})$ for all $\beta>1$.

Now suppose that 
$f:M\to M$ is an ergodic measure-preserving transformation
on a probability space $(M,\mu)$, and that $Y\subset M$ is measurable with $\mu(Y)>0$.
Suppose that $F:Y\to Y$ is a full-branch Gibbs-Markov map with respect to a probability measure $\mu_Y$ on $Y$, and that $\varphi:Y\to \Z^+$ is a return time, 
constant on partition elements, such that $F=f^\varphi$.
Form the tower $\Delta=Y^\varphi$ and tower map $f_\Delta:\Delta\to\Delta$.  The map $\pi_M:\Delta\to M$, $\pi_M(y,\ell)=f^\ell y$ defines a semiconjugacy between $f_\Delta$ and $f$.  We require moreover that $(\pi_M)_*\mu_\Delta=\mu$.
Then we say that $f$ is {\em modelled by a Young tower}.

\paragraph{Chernov-Markarian-Zhang structure}
Suppose that $(M,\mu)$ is a probability space and let
$f:M\to M$ be an ergodic and mixing measure-preserving transformation.
Roughly speaking, the map $f$ admits 
a Chernov-Markarian-Zhang structure if there is an integrable first return time $h:X\to\Z^+$ such that
the first return map $f_X=f^h:X\to X$ is modelled by a Young tower $Y^\sigma$.
The full map $f:M\to M$ is also modelled by a Young tower $Y^\varphi$.
We denote these towers by $\Delta=Y^\varphi$ and $\DE=Y^\sigma$ since in the applications that we have in mind either the tower $\DE$ is exponential or 
for any $q>1$ the subset $Y\subset X$ can be chosen such that $f_X$ is modelled by a Young tower $Y^\sigma$ with
$\mu_Y(\sigma>n)=O(n^{-q})$.  In the latter case, we say that $f_X$ is {\em modelled by Young towers with superpolynomial tails}.

In more detail, suppose $Y\subset X\subset M$ are Borel sets with 
$\mu(Y)>0$.
Define
the first return time $h:X\to\Z^+$ and first return map $f_X=f^h:X\to X$. 

We assume that $f_X:X\to X$ is modelled by a Young tower $\DE=Y^\sigma$ with 
return time $\sigma:Y\to\Z^+$ and return map $F=f_X^\sigma:Y\to Y$.
In particular, $F=f_X^\sigma:Y\to Y$ is a full-branch Gibbs-Markov map with ergodic invariant probability measure $\mu_Y$ and partition $\alpha$ such that
$\sigma$ is constant on partition elements.  
We require in addition that $h$ is constant on $f_X^\ell a$
for all $a\in\alpha$, $0\le\ell\le\sigma(a)-1$.

Define the {\em induced return time}
\begin{equation} \label{eq:induce}
\textstyle \varphi=h_\sigma:Y\to\Z^+, \qquad
\varphi(y)=\sum_{\ell=0}^{\sigma(y)-1}h(f^\ell y).
\end{equation}
Then $\varphi$ is integrable with respect to $\mu_Y$ and constant on partition elements.
In particular, $f:M\to M$ is modelled by a 
Young tower $\Delta=Y^\varphi$ with the same Gibbs-Markov map $F=f_X^\sigma=f^\varphi$.

We say that $f:M\to M$ satisfying these assumptions possesses a {\em Chernov-Markarian-Zhang structure}.

\begin{rmk} \label{rmk:CMZ}
The method of choosing a first return map modelled by a Young tower with exponential tails arises in various contexts in the literature, see for example~\cite{BruinLuzzattoStrien03,BruinTerhesiu18} in the noninvertible context.
However, the method plays a special role in the context of billiards~\cite{ChernovZhang05,Markarian04}, see Remark~\ref{rmk:CMZ2} below.
\end{rmk}

\begin{rmk} \label{rmk:d}
It is part of our set up that $\mu$ is mixing, but in general the tower map $f_\Delta:\Delta\to\Delta$ is mixing only up to a finite cycle $d\ge1$ where $d$ is often unknown.  As in~\cite[Theorem~2.1, Proposition~10.1]{Chernov99}, the {\em a priori} knowledge  that $\mu$ is mixing ensures that for many purposes the value of $d$ is irrelevant (in fact it suffices that $\mu$ is ergodic for all powers of $f$).  
\end{rmk}

\paragraph{Dynamically H\"older observables}

Suppose that $f:M\to M$ possesses a Chernov-Markarian-Zhang structure as above.
Fix $\theta\in(0,1)$.  For $v:M\to\R$, define
\begin{equation} \label{eq:dyn}
{\|v\|}_\cH=|v|_\infty+{|v|}_\cH, \quad 
{|v|}_\cH=\sup_{y,y'\in Y,\,y\neq y'}\sup_{0\le\ell\le \varphi(y)-1}\frac{|v(f^\ell y)-v(f^\ell y')|}{\theta^{s(y,y')}}.
\end{equation}
We say that $v$ is {\em dynamically H\"older}
if ${\|v\|}_\cH<\infty$ and denote by $\cH(M)$ the space of such observables.

Of particular interest are observables supported in $X$.
We identify $L^\infty(X)$ with $\{w\in L^\infty(M):w|_{M\setminus X}\equiv0\}$.
Also, we write $\cH(X)=\{v\in \cH(M):\supp v\subset X\}$.

It is standard that H\"older observables are dynamically H\"older for the classes of dynamical systems of interest in this paper, as we now recall.
Given $\eta\in(0,1]$ and a metric $d$ on $M$, define
\[
{|v|}_{C^\eta}=\sup_{x,x'\in M,\,x\neq x'}|v(x)-v(x)|/d(x,x')^\eta.
\]
Let $C^\eta(M)$ be the space of bounded observables $v:M\to\R$ for which
${|v|}_{C^\eta}<\infty$.

\begin{prop} \label{prop:dyn}
Let $\eta\in(0,1]$.  Suppose that
there exist $K>0$, $\theta_0\in(0,1)$ such that
$d(f^\ell y,f^\ell y')\le K\theta_0^{s(y,y')}$
 for all $y,y'\in Y$, $0\le\ell\le \varphi(y)-1$.

Then 
$C^\eta(M)\subset \cH(M)$ where we may choose any $\theta\in[\theta_0^\eta,1)$.
\end{prop}

\begin{proof}
Let $v\in C^\eta(M)$, $y,y'\in Y$, $0\le\ell<\varphi(y)-1$. Then
\begin{align*}
|v(f^\ell y)- v(f^\ell y')|
 \le {|v|}_{C^\eta}d(f^\ell y,f^\ell y')^\eta
   \le K^\eta {|v|}_{C^\eta} \theta^{s(y,y')}.
\end{align*}
Hence ${|v|}_{\cH}\le K^\eta {|v|}_{C^\eta}$
and it follows that $v\in \cH(M)$.
\end{proof}

\section{Statement of the main result}
\label{sec:main}

In this section, we state our main abstract result for
maps $f:M\to M$ with a Chernov-Markarian-Zhang structure.  
Let $Y\subset X\subset M$ denote the corresponding return map sets and
recall that $\varphi=h_\sigma:Y\to\Z^+$ is the induced return time.
Throughout, we suppose that 
$\mu_Y(\varphi>n)=O(n^{-\beta'})$ for some $\beta'>1$.
(As discussed in Section~\ref{sec:tail}, in our main examples any $\beta'<\beta$ is permitted where
$\mu_X(h>n)=O(n^{-\beta})$, and often we can take $\beta'=\beta$. However, $h$ and $\beta$ play no role in this section.)

Define  the correlation function $\rho_{v,w}(n)$ as in~\eqref{eq:rho}.
It follows from Young~\cite{Young99} that $\rho_{v,w}(n)=O(\|v\|_\cH |w|_\infty\, n^{-(\beta'-1)})$ for all $v\in\cH(M)$, $w\in L^\infty(M)$, $n\ge1$.
We can now state our main theorem.
Let 
\begin{equation} \label{eq:gamma}
\textstyle \sigma_n= \int_Y \sigma 1_{\{\varphi>n\}}d\mu_Y, \qquad
\gamma_n=n^{-\beta'}\star \sigma_n.
\end{equation}
Define $\zeta_{\beta'}$ as in~\eqref{eq:zeta}.

\begin{thm} \label{thm:main}
Let $f:M\to M$ be a map with a Chernov-Markarian-Zhang structure, and suppose that $\mu_Y(\varphi>n)=O(n^{-\beta'})$ for some $\beta'>1$.  
Then there  is a constant $C>0$ such that for all $n\ge1$,
\begin{itemize}
\item[(a)]
$\displaystyle \Big|\rho_{v,w}(n)-\bvarphi^{-1}\sum_{j>n} \mu_Y(\varphi>j)\int_Mv\,d\mu\int_Mw\,d\mu\Big|
\le C{\|v\|}_{\cH}|w|_\infty(\gamma_n+\zeta_{\beta'}(n))$
\newline for all $v\in \cH(X)$, $w\in L^\infty(X)$,
\item[(b)]
$|\rho_{v,w}(n)|\le C{\|v\|}_{\cH}|w|_\infty\gamma_n$ for all
$v\in \cH(X)$ with $\int_M v\,d\mu=0$, $w\in L^\infty(M)$.
\end{itemize}
\end{thm}

Clearly $n^{-\beta'}\le\sigma_n\le\gamma_n$.
The sequences are readily estimated from above:
\begin{prop} \label{prop:int}
\begin{itemize}

\parskip= -2pt
\item If $\sigma$ is bounded, then $\gamma_n=O( n^{-\beta'})$.
\item If $\sigma$ has exponential tails, then $\gamma_n=O( n^{-\beta'}\log n)$.
\item Let $\eps>0$.  If $\mu_Y(\sigma>n)=O(n^{-q})$ with $q$ sufficiently large (depending on $\eps$),  then $\gamma_n=O( n^{-(\beta'-\eps)})$.
\end{itemize}
\end{prop}

\begin{proof}
Suppose that $\sigma$ has exponential tails, and fix $K>0$.  Then 
\[
\sigma_n=\int_Y \sigma 1_{\{\varphi>n\}}\mu_Y
\le K\mu_Y(\varphi>n) + \int_Y 1_{\{\sigma>K\}}\sigma\,d\mu_Y
\ll Kn^{-\beta'}+O(e^{-cK}),
\]
for some $c>0$.  Choosing $K=(\beta'/c)\log n$, we obtain
$\sigma_n=O(n^{-\beta'}\log n)$.
Also $n^{-\beta'}\star n^{-\beta'}\log n =O(n^{-\beta'}\log n)$.

The other cases are similar and hence omitted.
\end{proof}

\begin{rmk} \label{rmk:main}
In particular, if $\sigma$ is bounded, then we are back in the situation of~\cite{Gouezel04a,Sarig02} and our estimates reduce to theirs.
Note that we have the slight improvement in Theorem~\ref{thm:main}(b) that $w$ is an arbitrary $L^\infty$ function, not necessarily supported in $X$.
Such a result does not seem to have been noted before.

When $\sigma$ is unbounded,~\cite{Gouezel04a,Sarig02} does not apply directly since the estimates required for applying operator renewal theory are problematic on $X$, while the dynamics on $Y$ is not given by a first return map, so it is necessary to incorporate arguments from~\cite{BruinTerhesiu18}.
\end{rmk}

\begin{rmk} 
As in~\cite{BruinTerhesiu18}, we can incorporate observables supported on the whole of $M$ that decay sufficiently quickly off $X$.  
Let $\tilde\sigma:Y\to\R$.
Suppose that $v,w:M\to\R$ are such that
$\sum_{\ell=0}^{\varphi-1}|v\circ f^\ell|\le \tilde\sigma$,
$\sum_{\ell=0}^{\varphi-1}|w\circ f^\ell|\le \tilde\sigma$ on $Y$.
Then Theorem~\ref{thm:main} holds with $\sigma_n$ defined using $\tilde\sigma$ instead of $\sigma$.

In contrast to~\cite{BruinTerhesiu18}, we do not require that the H\"older constants of $v$ decay off $X$.
\end{rmk}

\section{Proof of the main theorem}
\label{sec:proof}

In this section we prove Theorem~\ref{thm:main}.
We continue to suppose that $f:M\to M$  possesses a Chernov-Markarian-Zhang structure and that $\mu_Y(\varphi>n)=O(n^{-\beta'})$ for some $\beta'>1$.  
In Subsection~\ref{sec:Delta}, we state an analogous result, Theorem~\ref{thm:Delta}, for observables defined on $\Delta$ and deduce Theorem~\ref{thm:main} as a consequence.
In Subsection~\ref{sec:TR}, we recall some results from operator renewal theory for the Gibbs-Markov map on $Y$.
In Subsection~\ref{sec:X}, we prove Theorem~\ref{thm:Delta}.

\subsection{Tower reformulation}
\label{sec:Delta}

Recall that $M$ is modelled by a Young tower $\Delta=Y^\varphi$ and that
$F=f^\varphi:Y\to Y$ is a full-branch Gibbs-Markov map.
Let $d=\gcd\{\varphi(a):a\in\alpha\}$.  
The tower map $f_\Delta:\Delta\to\Delta$ is mixing if and only if  
$d=1$.   To deal with the cases $d=1$ and $d\ge2$ uniformly, we set $\Phi=d^{-1}\varphi$.  
Replace $\Delta$ by $\Delta=Y^\Phi$ and redefine $f_\Delta:\Delta\to\Delta$ accordingly.  Also, define $\mu_\Delta=(\mu_Y\times{\rm counting})/\bPhi$.
Then $f_\Delta$ is mixing.

Define
\[
\pi_M:\Delta\to M, \qquad \pi_M(y,\ell)=f^{d\ell}y.
\]
Then $\pi_M$ is a semiconjugacy between $f_\Delta$ and $g=f^d$,
and $(\pi_M)_*\mu_\Delta$   is an ergodic $g$-invariant probability measure on $M$.
It is an easy consequence of the definitions that $(\pi_M)_*\mu_\Delta$ is absolutely continuous with respect to the original measure $\mu$.  Moreover, $\mu$ is mixing for $f$ by assumption and so
is also ergodic for $g$.   Hence $\pi_M$ is a measure-preserving semiconjugacy between $(\Delta,f_\Delta,\mu_\Delta)$ and $(M,g,\mu)$.

Observables $v:M\to\R$ supported in $X$  lift to observables $\tv=v\circ \pi_M:\Delta\to\R$ supported in $\pi_M^{-1}(X)\subset\Delta$.
More generally, we consider observables $v$ supported in
$X_d=X\cup f^{-1}X\cup\cdots\cup f^{-(d-1)}X$.
Such observables lift to observables $\tv:\Delta\to\R$ supported in $\hX=\pi_M^{-1}(X_d)\subset\Delta$.

\begin{prop} \label{prop:hX}
$\sum_{\ell=0}^{\Phi(y)-1}1_{\hX}(y,\ell) \le \sigma(y)$ for $y\in Y$.
\end{prop}

\begin{proof}
Let $y\in Y$.
Set $h_j(y)=\sum_{i=0}^{j-1}h(f_X^ix)$.
Since $h:X\to\Z^+$ is the first return time to $X$ under the map $f:M\to M$
we have that $f^\ell y\in X$ for some 
$\ell\ge 0$ precisely when
$\ell=h_j(y)$ for some $j\ge0$.
Since $\varphi(y)=h_{\sigma(y)}(y)$,
there are precisely $\sigma(y)$ returns of $y$ to $X$ under $f$ by time $\varphi(y)$.
Hence 
$\sum_{\ell=0}^{\varphi(y)-1}1_{\{f^\ell y\in X\}} = \sigma(y)$ for $y\in Y$.

Finally,
\[
\sum_{\ell=0}^{\Phi(y)-1}1_{\hX}(y,\ell) = 
\sum_{\ell=0}^{\Phi(y)-1}1_{\{f^{\ell d}y\in X_d\}}
\le \sum_{\ell=0}^{\Phi(y)-1}\sum_{i=0}^{d-1}1_{\{f^{\ell d+i}y\in X\}}
=\sum_{\ell=0}^{\varphi(y)-1}1_{\{f^\ell y\in X\}},
\]
and the result follows.
\end{proof}

Fix $\theta\in(0,1)$ and define
\begin{equation} \label{eq:norm1}
\|\tv\|_\theta=|\tv|_\infty+|\tv|_\theta, \quad 
|\tv|_\theta
 =\sup_{y,y'\in Y,\,y\neq y'}\sup_{0\le\ell\le \Phi(y)-1}\frac{|\tv(y,\ell)-\tv(y',\ell)|}{\theta^{s(y,y')}}.
\end{equation}
Let $\cF_\theta(\hX)$ denote the space of observables $\tv$ supported in $\hX$  with $\|\tv\|_\theta<\infty$.

Given $\tv,\tw\in L^\infty(\Delta)$, we define 
\[
\textstyle 
\rho^*_{\tv,\tw}(n)= \int_\Delta\tv\,\tw\circ f_\Delta^n\,d\mu_\Delta.
\]

\begin{thm} \label{thm:Delta}
There  is a constant $C>0$ such that for all $n\ge1$,
\begin{itemize}
\item[(a)]
$\Big|\rho^*_{\tv,\tw}(n)-(1+\bPhi^{-1}\sum_{j>n} \mu_Y(\Phi>j))\int_\Delta\tv\,d\mu_\Delta\int_\Delta\tw\,d\mu_\Delta\Big|
\le C\|\tv\|_\theta|\tw|_\infty(\gamma_{nd}+\zeta_{\beta'}(n))$
for all $\tv\in \cF_\theta(\hX)$, $\tw\in L^\infty(\hX)$,
\item[(b)]
$|\rho^*_{\tv,\tw}(n)|\le C\|\tv\|_\theta|\tw|_\infty\gamma_{nd}$ for all
$\tv\in \cF_\theta(\hX)$ with $\int_\Delta \tv\,d\mu_\Delta=0$, $\tw\in L^\infty(\Delta)$.
\end{itemize}
\end{thm}

We conclude this subsection by showing that Theorem~\ref{thm:main} is a direct consequence of Theorem~\ref{thm:Delta}.

\begin{pfof}{Theorem~\ref{thm:main}}
Recall that $g=f^d$.
Write $n=md-r$ for $m\ge1$, $0\le r\le d-1$.  
Using the measure-preserving semiconjugacy
$\pi_M:\Delta\to M$,
$\pi_M(y,\ell)=g^\ell y$, 
we can write 
\begin{align} \label{eq:main} \nonumber
\rho_{v,w}(n) &\textstyle =\int_M v\circ f^r\,w\circ g^m\,d\mu-\int_M v\,d\mu\int_M w\,d\mu
\\[.75ex] &\textstyle =\rho^*_{\widetilde{v\circ f^r},\tw}(m)-\int_\Delta \tv\,d\mu_\Delta\int_\Delta \tw\,d\mu_\Delta.
\end{align}

Suppose as in part~(a) that $v\in \cH(X)$ and $w\in L^\infty(X)$.
Then $\supp \widetilde{v\circ f^r}\subset \pi_M^{-1} f^{-r}X\subset \pi_M^{-1} X_d=\hX$ and
$\supp \tw\subset \pi_M^{-1}X\subset X_d$.  
Moreover, 
 \begin{align*}
|\widetilde{v\circ f^r}|_\theta
 & =\sup_{y,y'\in Y,\,y\neq y'}\sup_{0\le\ell\le \Phi(y)-1}|v(f^{d\ell+r}y)-v(f^{d\ell+r}y')|/\theta^{s(y,y')}
 \\ & \le \sup_{y,y'\in Y,\,y\neq y'}\sup_{0\le\ell\le \varphi(y)-1}|v(f^\ell y)-v(f^\ell y')|/\theta^{s(y,y')}
= {|v|}_{\cH}.
\end{align*}
Hence it follows from~\eqref{eq:main} and Theorem~\ref{thm:Delta}(a) that
\begin{align*}
 \Big|\rho_{v,w}(n) & \textstyle -\bPhi^{-1}\sum_{j>m} \mu_Y(\Phi>j)\int_Mv\,d\mu\int_Mw\,d\mu\Big|
 \\ 
& =  \textstyle \Big|\rho^*_{\widetilde{v\circ f^r},\tw}(m)-(1+\bPhi^{-1}\sum_{j>m} \mu_Y(\Phi>j))\int_\Delta\tv\,d\mu_\Delta\int_\Delta\tw\,d\mu_\Delta\Big|
\\ & \ll\|\widetilde{v\circ f^r}\|_\theta|\tw|_\infty(\gamma_{md}+\zeta_{\beta'}(m))
 \le{\|v\|}_{\cH}|w|_\infty(\gamma_{md}+\zeta_{\beta'}(m)).
\end{align*}
Now $\gamma_{md}=\gamma_{n+r}\le\gamma_n$ and $\zeta_{\beta'}(m)\ll \zeta_{\beta'}(n)$.
Moreover,
\[
\bPhi^{-1}\sum_{j>m} \mu_Y(\Phi>j)=\bvarphi^{-1}d\sum_{j>(n+r)/d}\mu_Y(\varphi>jd)
=\bvarphi^{-1}d\sum_{j>[n/d]}\mu_Y(\varphi>jd)+O(n^{-\beta'}).
\]
By monotonicity of $\mu_Y(\varphi>k)$,
\[
\sum_{k=jd}^{jd+d-1}\mu_Y(\varphi>k)\le
d\mu_Y(\varphi>jd)\le \sum_{k=jd-d}^{jd-1}\mu_Y(\varphi>k).
\]
Summing over $j$ yields
\[
\sum_{k\ge ([n/d]+1)d}\mu_Y(\varphi>k)\le d\sum_{j>[n/d]}\mu_Y(\varphi>jd)
\le \sum_{k\ge [n/d]d}\mu_Y(\varphi>k),
\]
and hence $d\sum_{j>[n/d]}\mu_Y(\varphi>jd)=\sum_{k>n}\mu_Y(\varphi>k)+O(n^{-\beta'})$.
This completes the proof of part~(a).

Similarly, in the context of part~(b),
\[
|\rho_{v,w}(n)|=|\rho^*_{\widetilde{v\circ f^r},\tw}(m)|
\ll \|\widetilde{v\circ f^r}\|_\theta |\tw|_\infty \gamma_{md}
\le {\|v\|}_{\cH}|w|_\infty \gamma_n
\]
by Theorem~\ref{thm:Delta}(b).
\end{pfof}

\subsection{Operator renewal theory on $Y$}
\label{sec:TR}

Set $\D=\{z\in\C:|z|<1\}$ and $\barD=\{z\in\C:|z|\le1\}$.
Let $R:L^1(Y)\to L^1(Y)$ and $L:L^1(\Delta)\to L^1(\Delta)$ denote the transfer operators for $F:Y\to Y$ and $f_\Delta:\Delta\to\Delta$ respectively.
Define the renewal operators $R(n),\,T(n):L^1(Y)\to L^1(Y)$, 
\[
R(n)v=R(1_{\{\Phi=n\}}v),\;n\ge1,\qquad
T(n)v=1_Y L^n(1_Y v),\,n\ge0,
\]
and the corresponding Fourier series
$\hR(z),\,\hT(z):L^1(Y)\to L^1(Y)$, for $z\in\D$,
\[
\hR(z)=\sum_{n=1}^\infty R(n)z^n,\qquad
\hT(z)=\sum_{n=0}^\infty T(n)z^n.
\]
A calculation shows that 
\[
\hR(z)v=R(z^\Phi v)\;
\text{for $z\in\barD\quad$ and}\quad \hT=(I-\hR)^{-1}\;\text{on $\D$}.
\]
Also, for $z\in\D$, we define
\[
\hB(z):L^1(Y)\to L^1(Y), \qquad \hB(z)=(z-1)\hT(z),
\]
with Fourier coefficients $B(n)$, $n\ge0$.

Given $v:Y\to\R$, define $|v|_\theta=\sup_{y\neq y'}|v(y)-v(y')|/\theta^{s(y,y')}$ and
$\|v\|_\theta=|v|_\infty+|v|_\theta$.
Let $\cF_\theta(Y)$ be the Banach space of observables $v$ with
$\|v\|_\theta<\infty$.
Since $F:Y\to Y$ is a Gibbs-Markov map
and $\Phi:Y\to\Z^+$ is constant on partition elements,
the operators $R$, $R(n)$, $T(n)$, $\hR(z)$ and $\hT(z)$ are bounded operators on $\cF_\theta(Y)$.
Define $Pv=\bPhi^{-1}\int_Y v\,d\mu_Y$.

Define $\zeta_{\beta'}$ as in the introduction.
Since $F$ is mixing and
$\gcd\{\Phi(a):a\in\alpha\}=1$, it follows from~\cite{Gouezel04a} that
on $\cF_\theta(Y)$:
\begin{lemma} \label{lem:Tn}
\begin{itemize}
\item[(a)] $T(n)=b(n)P+H(n)$
where $b(n)=1+\bPhi^{-1}\sum_{j>n}\mu_Y(\Phi>j)$ and
$\|H(n)\|_\theta=O(\zeta_{\beta'}(n))$.
\item[(b)] There is a sequence $\tilde b(n)>0$ such that
$T(n)=\tilde b(n)P+\tH(n)$ where $\|\tH(n)\|_\theta=O(n^{-\beta'})$.
\item[(c)] 
$\|B(n)\|_\theta=O(n^{-\beta'})$.
\qed
\end{itemize}
\end{lemma}

\subsection{Proof of Theorem~\ref{thm:Delta}}
\label{sec:X}

Let $\tv,\tw\in L^\infty(\Delta)$.  Define $V(n),\,W(n):Y\to\R$, 
\[
V(n)(y)=1_{\{\Phi(y)\ge n\}}\tv(y,\Phi(y)-n),\,n\ge1,
\qquad
W(n)(y)= 1_{\{\Phi(y)>n\}}\tw(y,n), \, n\ge0,
\]
as well as
\[
J_0(n)  =\int_\Delta 1_{\{n+\ell< \Phi(y)\}}\tv(y,\ell)\tw(y,n+\ell)\,d\mu_\Delta.
\]
The following formula is a discrete time analogue of a formula in~\cite{Pollicott85}.
(The proof is in the Appendix.)

\begin{prop} \label{prop:conv}
$\rho^*_{\tv,\tw}(n)=J_0(n)+\bar\Phi^{-1}\int_Y
(T(n)\star RV(n))\star W(n)\,d\mu_Y$
for all $\tv,\tw\in L^\infty(\Delta)$, $n\ge1$.
\footnote{Extending the convention mentioned in the introduction, $T(n)\star RV(n)=\sum_{j=0}^n T(j)RV(n-j)$.
Similarly, $(T(n)\star RV(n))\star W(n)=\sum_{j=0}^n G(j)W(n-j)$ where
$G(n)=T(n)\star RV(n)$.}  \qed
\end{prop}

\begin{prop} \label{prop:Vn}
$|V(n)|_1\le  |\tv|_\infty\, \mu_Y(\Phi\ge n)$ and
$|W(n)|_1\le  |\tw|_\infty\, \mu_Y(\Phi>n)$
for all $\tv,\tw\in L^\infty(\Delta)$, $n\ge1$.
Moreover, there is a constant $C>0$ such that
$\|RV(n)\|_\theta\le C \|\tv\|_\theta\, \mu_Y(\Phi\ge n)$ 
for all $\tv\in\cF_\theta(\Delta)$, $n\ge1$.
\end{prop}

\begin{proof}
The estimates for $|V(n)|_1$ and $|W(n)|_1$ are immediate.
Let $y,y'\in a$, $a\in\alpha$.  Then
$|V(n)(y)|
 \le 1_{\{\Phi(a)\ge n\}}|\tv|_\infty$
and
\begin{align*}
|V(n)(y)-V(n)(y')| & = 1_{\{\Phi(a)\ge n\}}
|\tv(y,\Phi(a)-n)-\tv(y',\Phi(a)-n)|
\\ & \le 1_{\{\Phi(a)\ge n\}}|\tv|_\theta\,\theta^{s(y,y')}.
\end{align*}

Given $y\in Y$, set $y_a=F^{-1}(y)\cap a$.  Then
$(RV(n))(y)=\suma \xi(y_a)V(n)(y_a)$ so
\[
|RV(n)|_\infty\ll
\suma \mu_Y(a)|\tv|_\infty\, 1_{\{\Phi(a)\ge n\}}
=|\tv|_\infty\mu_Y(\Phi\ge n),
\]
by~\eqref{eq:GM}.
Also, for $y,y'\in Y$,
\begin{align*}
|(RV(n))(y)- & (RV(n))(y')|   
\\ & \le  
 \suma \big\{\xi(y_a)|V(n)(y_a)-V(n)(y'_a)|
+|\xi(y_a)-\xi(y'_a)||V(n)(y'_a)|\big\}
\\[.75ex] 
& \ll \theta^{s(y,y')}\|\tv\|_\theta \suma\mu_Y(a)1_{\{\Phi(a)\ge n\}}
= \theta^{s(y,y')}\|\tv\|_\theta\, \mu_Y(\Phi\ge n).
\end{align*}
Hence $|RV(n)|_\theta\ll \|\tv\|_\theta\, \mu_Y(\Phi\ge n)$ and the estimate for $\|RV(n)\|_\theta$ follows.  
\end{proof}

\begin{cor} \label{cor:Hn}
Let $H(n),\,\tH(n):\cF_\theta(Y)\to\cF_\theta(Y)$ be as Lemma~\ref{lem:Tn}.
There is a constant $C>0$ such that
 \begin{itemize}
\item[(a)] $|\int_Y(H(n)\star RV(n))\star W(n)\,d\mu_Y|  \le C \|\tv\|_\theta |\tw|_\infty\, \zeta_{\beta'}(n)$,
\item[(b)] $|\int_Y(\tH(n)\star RV(n))\star W(n)\,d\mu_Y|  \le C \|\tv\|_\theta |\tw|_\infty\, n^{-\beta'}$,
\end{itemize}
for all $\tv\in\cF_\theta(\Delta)$, $\tw\in L^\infty(\Delta)$, $n\ge1$.
\end{cor}

\begin{proof}
By Proposition~\ref{prop:Vn} and the estimate for $H(n)$ in Lemma~\ref{lem:Tn}(a),
the first integral is estimated by
\begin{align*}
|H(n)\,\star\, RV(n)|_\infty\,\star\, |W(n)|_1
 & \le \|H(n)\|_\theta\,\star\,\|RV(n)\|_\theta\,\star\, |W(n)|_1
 \\ & \ll \zeta_{\beta'}(n)\,\star\, \|\tv\|_\theta\, n^{-\beta'}\star |\tw|_\infty\, n^{-\beta'}
\ll \|\tv\|_\theta |\tw|_\infty \,\zeta_{\beta'}(n).
\end{align*}
The second integral is estimated in the same way using Lemma~\ref{lem:Tn}(b).
\end{proof}

For $n\ge0$, define
\[
A_1(n)(y)=1_{\{\Phi(y)>n\}}\sum_{\ell=0}^{\Phi(y)-n-1}\tv(y,\ell), \qquad
A_2(n)(y)=\sum_{\ell=0}^{\Phi(y)-1}1_{\{n< \ell\}}\tw(y,\ell).
\]

\begin{lemma} \label{lem:JA}  
\begin{itemize}
\item[(a)] 
$|J_0(n)|\le \bPhi^{-1}|\tv|_\infty|\tw|_\infty \,\sigma_{nd}$
for all $\tv\in L^\infty(\hX)$, $\tw\in L^\infty(\Delta)$, 
\mbox{$n\ge1$}.
\item[(b)]
$|A_1(n)|_1 \le |\tv|_\infty\, \sigma_{nd}$
for all $\tv\in L^\infty(\hX)$, $n\ge1$.
\item[(c)]
$|A_2(n)|_1 \le |\tw|_\infty\, \sigma_{nd}$
for all $\tw\in L^\infty(\hX)$, $n\ge1$.
\end{itemize}
\end{lemma}

\begin{proof}
Since $\tv$ is supported in $\hX$,
\begin{align*}
|J_0(n)|
& \le |\tv|_\infty|\tw|_\infty\int_{\Delta} 1_{\{n< \Phi(y)\}}1_{\hX}(y,\ell)\,d\mu_\Delta
\\ & = \bPhi^{-1}|\tv|_\infty|\tw|_\infty\int_Y 1_{\{\varphi(y)>nd\}}\sum_{\ell=0}^{\Phi(y)-1}1_{\hX}(y,\ell)\,d\mu_Y
\\ & \le \bPhi^{-1}|\tv|_\infty|\tw|_\infty\int_Y 1_{\{\varphi>nd\}}\sigma\,d\mu_Y
= \bPhi^{-1}|\tv|_\infty|\tw|_\infty
\sigma_{nd},
\end{align*}
by Proposition~\ref{prop:hX}.

 Next, 
$|A_1(n)(y)|\le 1_{\{\Phi(y)>n\}}\sum_{\ell=0}^{\Phi(y)-1}|\tv(y,\ell)|$
and
it again follows that
$\int_Y |A_1(n)|\,d\mu_Y \le |\tv|_\infty\int_Y 1_{\{\Phi>n\}}\sigma\,d\mu_Y
=|\tv|_\infty\,\sigma_{nd}$.
Similarly for $A_2(n)$.
\end{proof}

The Fourier series for $V(n)$, $W(n)$ are given by
\[
\hV(z)(y)=\sum_{\ell=0}^{\Phi(y)-1} z^{\Phi(y)-\ell}\tv(y,\ell), \qquad
\hW(z)(y)=\sum_{\ell=0}^{\Phi(y)-1} z^\ell\tw(y,\ell),\quad z\in\D.
\]

\begin{prop} \label{prop:AA}
$\hA_1(z)
=(z-1)^{-1}(\hV(z)-\hV(1))$,
and $\hA_2(z)= (z-1)^{-1}(\hW(z)-\hW(1))$ for $z\in\D$. 
\end{prop}

\begin{proof}
We have
\begin{align*}
\hA_1(z)(y) & =\sum_{n=0}^\infty z^nA_1(n)(y)
=\sum_{n=0}^{\Phi(y)-1} z^n\sum_{\ell=0}^{\Phi(y)-n-1}\tv(y,\ell)
=\sum_{\ell=0}^{\Phi(y)-1} \Big(\sum_{n=0}^{\Phi(y)-\ell-1}z^n\Big)\tv(y,\ell)
\\ & =(z-1)^{-1}\sum_{\ell=0}^{\Phi(y)-1}(z^{\Phi(y)-\ell}-1)\tv(y,\ell)
=(z-1)^{-1}(\hV(z)(y)-\hV(1)(y)).
\end{align*}
The calculation for $A_2$ is similar.
\end{proof}

\begin{pfof}{Theorem~\ref{thm:Delta}}
By Lemma~\ref{lem:Tn}(a) and Proposition~\ref{prop:conv},
\[
\rho^*_{\tv,\tw}(n)=J_0(n)+E(n)
+ \bPhi^{-1}\int_Y (H(n)\star RV(n)) \star W(n)\,d\mu_Y,
\]
where
\[
E(n)=\bPhi^{-1}\int_Y (b(n) P\star RV(n))\star W(n)\,d\mu_Y,
\qquad b(n)=1+\bPhi^{-1}\sum_{j>n}\mu_Y(\Phi>j).
\]
By Corollary~\ref{cor:Hn}(a) and Lemma~\ref{lem:JA}(a), 
\begin{equation} \label{eq:proof}
\rho^*_{\tv,\tw}(n)=E(n)+O\big(\|\tv\|_\theta|\tw|_\infty(\sigma_{nd}+\zeta_{\beta'}(n))\big).
\end{equation}

Now, $E(n)= b(n)\star PV(n)\star PW(n)$ so
\begin{align*}
\hE(z) 
& = \hat b(z)P\hV(z)P\hW(z) \\
& =\hat b(z)\Big\{P \hV(1)P \hW(1) 
+P\hV(1)P(\hW(z)-\hW(1)) 
+P(\hV(z)-\hV(1))P \hW(z) \Big\}
\\ & 
=\hat b(z)P\hV(1) P\hW(1)
+(z-1)\hat b(z)\big\{ P\hV(1) P\hA_2(z)
+P\hA_1(z)P\hW(z)\big\}.
\end{align*}
Moreover, $(z-1)\hat b(z)=-b(1)z+\sum_{n=2}^\infty (b(n-1)-b(n))z^n$
with Fourier coefficients that are $O(\mu_Y(\Phi>n))$.
Hence it follows from Proposition~\ref{prop:Vn} and
Lemma~\ref{lem:JA} that
\[
E(n)  = b(n)P\hV(1)P\hW(1)+ O(|\tv|_\infty|\tw|_\infty\, n^{-\beta'}\star\sigma_{nd}),
\]
and we obtain
\[
\rho^*_{\tv,\tw}(n)=b(n)P\hV(1)P\hW(1)+O\big(\|\tv\|_\theta|\tw|_\infty(\gamma_{nd}+\zeta_{\beta'}(n))\big).
\]

Also, 
\[
P\hV(1)=\bPhi^{-1}\int_Y \hV(1)\,d\mu_Y
=\bPhi^{-1}\int_Y \sum_{\ell=0}^{\Phi(y)-1}\tv(y,\ell)\,d\mu_Y(y)
=\int_\Delta \tv\,d\mu_\Delta,
\]
and similarly $P\hW(1)=\int_\Delta \tw\,d\mu_\Delta$.
This completes the proof of part~(a).

The proof of part~(b) proceeds in much the same way but with $b(n)$ and $H(n)$ replaced by $\tilde b(n)$ and $\tH(n)$ from Lemma~\ref{lem:Tn}(b).  Using Corollary~\ref{cor:Hn}(b) instead of
Corollary~\ref{cor:Hn}(a), we obtain 
\[
\rho^*_{\tv,\tw}(n)=\tE(n)+ O(\|\tv\|_\theta|\tw|_\infty\,\sigma_{nd}),
\]
where $\tE(n)=\tilde b(n)\star PV(n)\star PW(n)$.
Calculating as in part~(a) and using 
$P\hV(1)=0$,
\[
\widehat{\tE}(z)=(z-1)\widehat{\tilde b}(z)P\hA_1(z)P\hW(z).
\]
By Lemma~\ref{lem:Tn}(b),
$(z-1)\widehat{\tilde b}(z)P=\hB(z)-(z-1)\widehat{\tH}(z)$ and hence has Fourier coefficients $h(n)$ that satisfy $|h(n)|=O(n^{-\beta'})$ by
Lemma~\ref{lem:Tn}(b,c).
It follows that 
\[
\tE(n)=h(n)\star PA_1(n)\star PW(n) = O(|\tv|_\infty|\tw|_\infty\gamma_{nd}),
\]
yielding the desired estimate in part~(b).
Note also that the terms involving $A_2(n)$ are no longer present.  The estimate for $A_2(n)$ in Lemma~\ref{lem:JA}(c) was the only one that required $\tw$ to be supported in $\hX$, so part~(b) holds for all $\tw\in L^\infty(\Delta)$.
\end{pfof}

\section{Tail estimates}
\label{sec:tail}

In applications, we are often given information about the first return time $h:X\to\Z^+$.  To apply Theorem~\ref{thm:main}, it is necessary to translate this into information about the tails $\mu_Y(\varphi>n)$ of the induced return time $\varphi:Y\to\Z^+$.

We begin with a rough estimate of this type.

\begin{prop} \label{prop:rough}
Fix $\beta,\,\eps>0$.

Suppose that $f:M\to M$ possesses a Chernov-Markarian-Zhang structure and that $\DE=Y^\sigma$ has exponential tails.
\begin{itemize}
\item[(a)] If $\mu_X(h>n)=O(n^{-\beta})$, then
$\mu_Y(\varphi>n)=O((\log n)^\beta n^{-\beta})$.
\item[(b)]
If $\mu_X(h>n)\ge cn^{-\beta}$ for some $c>0$, then there exists $c'>0$ such that
$\mu_Y(\varphi>n)\ge c'(\log n)^{-1}n^{-\beta}$.
\end{itemize}

 Now suppose that $\DE$ has polynomial tails with $\mu_Y(\sigma>n)=O(n^{-q})$ for $q$ sufficiently large (depending on $\beta$ and $\eps$).
\begin{itemize}
\item[(c)] If $\mu_X(h>n)=O(n^{-\beta})$, then
$\mu_Y(\varphi>n)=O(n^{-(\beta-\eps)})$.
\item[(d)]
If $\mu_X(h>n)\ge cn^{-\beta}$ for some $c>0$, then there exists $c'>0$ such that
$\mu_Y(\varphi>n)\ge c'n^{-(\beta+\eps)}$.
\end{itemize}
\end{prop}

\begin{proof} 
(a) This is proved in~\cite{ChernovZhang05,Markarian04}.

\noindent(b)
Let $\tilde h=h\circ\pi_X:\DE\to\Z^+$ where $\DE=Y^\sigma$
and $\pi_X:\DE\to X$ is the semiconjugacy $\pi_X(y,\ell)=f_X^\ell y$.
Then for any $K>0$,
\begin{align*}
& \bar\sigma\mu_X(h>n)  =\bar\sigma \mu_\DE(\tilde h>n)
=\int_Y\sum_{\ell=0}^{\sigma(y)-1}1_{\{\tilde h(y,\ell)>n\}}\,d\mu_Y(y)
\\ & \qquad \le \int_Y \sigma 1_{\{\varphi>n\}}\,d\mu_Y
= \int_Y \sigma 1_{\{\sigma\le K\log n\}} 1_{\{\varphi>n\}}\,d\mu_Y
+ \int_Y \sigma 1_{\{\sigma> K\log n\}} 1_{\{\varphi>n\}}\,d\mu_Y
\\ & \qquad \le (K\log n)\mu_Y(\varphi>n)+|\sigma|_2(\mu_Y(\sigma>K\log n))^{1/2}.
\end{align*}
We have $\mu_Y(\sigma>n)=O(e^{-an})$ for some $a>0$.
Fixing $K$ sufficiently large, 
\[
(\log n)\mu_Y(\varphi>n) \gg \mu_X(h>n)+O(n^{-K d/2})\gg n^{-\beta},
\]
so $\mu_Y(\varphi>n)\gg (\log n)^{-1}n^{-\beta}$.

\noindent(c,d)  These arguments are similar and hence omitted.
\end{proof}

Next, we consider a sharper estimate following~\cite{MVsub}.
First we collect some special cases of existing results about limit laws.
Assume that $f_X:X\to X$ is modelled by a Young tower $\DE=Y^\sigma$ with $\sigma\in L^2(Y)$.
In particular, $F:Y\to Y$ is Gibbs-Markov.
Let $\bar\sigma=\int_Y\sigma\,d\mu_Y$.

\begin{lemma}  \label{lem:CLT}
Let $\psi:X\to\R$ be integrable with $\int_X\psi\,d\mu_X=0$, and define $\psi_\sigma:Y\to\R$,
$\psi_\sigma(y)=\sum_{\ell=0}^{\sigma(y)-1}\psi(f_X^\ell y)$.
Let $G$ denote a nonconstant random variable.
Let $b_n=n^{1/\beta}$, $1<\beta<2$, or $b_n=(n\log n)^{1/2}$.  (In the latter case, set $\beta=2$.)
\begin{itemize}
\item[(a)] 
$b_n^{-1}\sum_{j=0}^{n-1}\psi\circ f_X^j\to_d G$ if and only if
$b_n^{-1}\sum_{j=0}^{n-1}\psi_\sigma\circ F^j\to_d \bar\sigma^{1/\beta} G$.
\item[(b)] Suppose that $\psi_\sigma$ is constant on elements of the partition
$\alpha$ for the Gibbs-Markov map $F$.
If $b_n^{-1}\sum_{j=0}^{n-1}\psi_\sigma\circ F^j\to_d \bar\sigma^{1/\beta}G$, then 
$\mu_Y(|\psi_\sigma|>n)\sim \bar\sigma c_0n^{-\beta}$ where $c_0>0$ is a constant given explicitly in terms of $G$.  
\end{itemize}
\end{lemma}

\begin{proof}
(a)  
Since $F$ is Gibbs-Markov, the condition $\sigma\in L^2$ ensures that
$n^{-1/2}(\sigma_n-n\bar\sigma)$ converges in distribution (to a possibly degenerate normal distribution) and hence that
$b_n^{-1}(\sigma_n-n\bar\sigma)$ converges in probability to zero.  The result now follows from~\cite[Theorem~A.1]{MVsub}.
(See~\cite{Gouezel07,Sarig06} for related results.)  
\\[.75ex]
(b) Again using that $F$ is Gibbs-Markov, this follows from~\cite[Theorem~1.5]{Gouezel10b}.
\end{proof}

\begin{cor} \label{cor:tails}
Let $G$, $b_n$, $\beta$ and $c_0$ be as in Lemma~\ref{lem:CLT}.
Suppose that $(\sum_{j=0}^{n-1}h\circ f_X^j-n\int_X h\,d\mu_X)/b_n\to_d G$.
Then $\mu_Y(\varphi>n)\sim \bar\sigma c_0n^{-\beta}$.  
\end{cor}

\begin{proof}
Since $\varphi=h_\sigma$, it follows from Lemma~\ref{lem:CLT}(a) that
$(\sum_{j=0}^{n-1}\varphi\circ F^j-n\int_Y \varphi\,d\mu_Y)/b_n\to_d \bar\sigma^{1/\beta}G$.
By Lemma~\ref{lem:CLT}(b),
$\mu_Y(\varphi>n)\sim \bar\sigma c_0n^{-\beta}$.
\end{proof}

\section{Piecewise smooth multidimensional nonMarkovian nonuniformly expanding maps}
\label{sec:NUE}

In this section, we show how to combine the methods in this paper with a result of Alves {\em et al.}~\cite{AFLV11} to treat a large class of multidimensional examples.  In particular, we obtain essentially optimal upper and lower bounds, as well as strong statistical properties, for Hu-Vaienti maps~\cite{HuVaienti09}.

\subsection{Existence of Chernov-Markarian-Zhang structures in arbitrary dimensions}
\label{sec:piecewise}

Let $M\subset\R^k$ be compact.
We consider local diffeomorphisms $f:M\to M$ with finitely many branches.
That is, there are disjoint open subsets
$U_1,\dots,U_K\subset M$ with $M=\bigcup_{i=1}^K \overline{U}_i$, and 
there exists $\eta\in(0,1)$ and 
for $i=1,\dots,K$ there exist
$U_i'\subset \R^k$ open with $\overline{U_i}\subset U_i'$ such that
$f|_{U_i}$ extends to a $C^{1+\eta}$-diffeomorphism from $U_i'$ onto its range.

Next we specify a  compact first return set $X\subset M$ with
$\overline{\Int X}=X$.
(We could take $X$ to be the closure of one of the $U_i$ but this need not be the case.)
For simplicity, we suppose that the boundaries of $U_1,\dots,U_K$ and $X$ are piecewise smooth (with finitely many pieces).  
Let $\cS_0\subset M$ denote the singularity set
$\cS_0=\partial X\cup\bigcup_{i=1}^K\partial U_i$ for $f$.

Now define the first return time $h:X\to\Z^+$ and first return map $f_X=f^h:X\setminus\cS\to X\setminus\cS$ with singularity set 
\[
\cS=\{x\in X:f^jx\in\cS_0\;\,\text{for some $j=0,\dots,h(x)-1$}\}\cup \{h=\infty\}.
\]
A result of Alves {\em et al.}~\cite{AFLV11} guarantees under very mild conditions that $f_X$ is modelled by Young towers with superpolynomial tails if and only if $f_X$ has superpolynomial decay of correlations.
We verify these conditions for a large class of nonuniformly expanding maps.

Define $X_m=\{x\in X\setminus\cS:h(x)=m\}$.
Let $\|\;\|$ denote the Euclidean norm on $\R^k$ and on $k\times k$ matrices.
We suppose that there are constants $\lambda\in(0,1)$, $\delta>0$ and $C,q>1$ such that
\begin{itemize}
\item[(i)] $\Leb(x\in X:\dist(x,\cS)\le\eps)\ll \eps$ for all $\eps\in(0,1)$.
\item[(ii)] $\|(Df_X(x))^{-1}\|\le\min\{\lambda,Cm^{-\delta}\}$ and $\|Df_X(x)\|\le Cm^q$ for all $x\in X_m$, $m\ge1$.
\item[(iii)] $\|(D(f^i)(f^jx))^{-1}\|\le Cm^q$ for all $x\in X_m$, $m\ge1$, and $i,j\ge0$ with $i+j\le m$.
\item[(iv)] $\|f^jx-f^jy\|\le C\|x-y\|^\delta m^q$
for all $x,y\in X_m$ with $\dist(x,y)<\dist(x,\cS)/2$, $m\ge1$, $0\le j<m$.
\end{itemize}

\begin{rmk} \label{rmk:NUE}  If $f$ is noncontracting ($\|Df(x)v\|\ge\|v\|$ for all $x\in M$, $v\in\R^d$), then~(iii) is automatic with $Cm^q$ replaced by $1$ and~(iv) is automatic by~(ii) with $\delta=1$.
\end{rmk}

\begin{lemma} \label{lem:NUE}
Suppose that $f:M\to M$ is a nonuniformly expanding map satisfying conditions~(i)--(iv).  Let $\mu$ be an absolutely continuous mixing $f$-invariant probability measure on $M$ and define $\mu_X=\mu(X)^{-1}\mu|_X$.   Suppose further that
\begin{itemize}
\item[(v)]
$d\mu_X/d\Leb\in L^r(X)$ for some $r>1$,
\item[(vi)] For all $C^\eta$ observables $v:X\to\R$, all $w\in L^\infty(X)$,
and all $p>0$, there is a constant $C>0$ such that
$\big|\int_X v\,w\circ f_X^n\,d\mu_X-\int_X v\,d\mu_X\int_X w\,d\mu_X\big|\le Cn^{-p}$ for all $n\ge1$, and
\item[(vii)] $\sum_m (\log m)\,\mu_X(X_m)<\infty$.
\end{itemize}

Then $f$ possesses a Chernov-Markarian-Zhang structure and the map $f_X:X\to X$ is modelled by Young towers with superpolynomial tails.
\end{lemma}

\begin{proof}
To prove that $f_X$ is modelled by Young towers with superpolynomial tails,
we apply~\cite[Theorem~C]{AFLV11}.  Since there are some small inaccuracies in the statement there, we refer to~\cite[Theorem~A.1]{AraujoMsub} for a corrected version.
It suffices to verify that $\mu_X$ is an expanding measure and to
verify conditions (C0)--(C3) in~\cite[Appendix~A]{AraujoMsub}.

By condition~(ii), 
 $\log \|(Df_X)^{-1}\|\le\log\lambda<0$ and
\begin{equation} \label{eq:iv}
\|Df_X(x)^{-1}v\|\ge C^{-1}m^{-q}
\quad\text{for all $x\in X_m$, $v\in T_xX$ with $\|v\|=1$,}
\end{equation}
so $|\log \|Df_X(x)^{-1}\|\,|\le \log C+q\log m$ on $X_m$.
By~(vii), $\log \|(Df_X)^{-1}\|$ is integrable with respect to $\mu_X$ 
and $\int_X \log \|(Df_X)^{-1}\|\,d\mu_X\le\log\lambda<0$.
This is the definition for $\mu_X$ to be an expanding measure.

Assumption (i)  is precisely condition (C0).

By (ii) and definition of $\cS$,
\begin{equation} \label{eq:ii}
\dist(x,\cS)\le Cm^{-\delta}\diam M
\quad\text{for all $x\in X_m$, $m\ge1$.}
\end{equation}
By~\eqref{eq:ii} and~(ii),
\[
\dist(x,\cS)\ll \lambda^{-1}\le \|Df_X(x)v\|\ll Cm^q\ll \dist(x,\cS)^{-q/\delta},
\]
for $x\in X_m$, $v\in T_xX$ with $\|v\|=1$,
verifying (C1).

For conditions (C2) and (C3), we consider a pair of points $x,y\in M\setminus\cS$ with $\dist(x,y)<\dist(x,\cS)/2$.  in particular, $x,y\in X_m$ for some $m\ge1$ and $f^jx,f^jy$ lie in common open sets $U_{i(j)}$ for each $0\le j\le m-1$.

On $X_m$, we have
$\log|\det Df_X|=\log|\det D(f^m)|
=\sum_{j=0}^{m-1}\log|\det(Df)|\circ f^j$, so
for $x,y\in X_m$ with $d(x,y)<d(x,\cS)/2$,
\begin{align*}
\big|\log|\det & Df_X(x)| -\log|\det Df_X(y)|\big|  
\\ & \le 
\sum_{j=0}^{m-1}\big|\log|\det (Df)(f^jx)|- \log|\det (Df)(f^jy)|\big|
\\ & \ll \sum_{j=0}^{m-1}\|f^jx-f^jy\|^\eta
\ll \|x-y\|^{\delta\eta}m^{1+q\eta} 
 \ll \|x-y\|^{\delta\eta}\dist(x,\cS)^{-(1+q\eta)/\delta},
\end{align*}
by~\eqref{eq:ii} and~(iv).
This verifies (C3).
Also, 
\begin{align*}
\big|\log\|Df_X(x)^{-1}\|- \log\|Df_X(y)^{-1}\|\big|
& \le \|Df_X(x)^{-1}-Df_X(y)^{-1}\|/\|Df_X(x)^{-1}\|
\\ & \ll m^q \|Df_X(x)^{-1}-Df_X(y)^{-1}\|
\end{align*}
by~\eqref{eq:iv}.
On $X_m$ we have $(Df_X)^{-1}=(D(f^m))^{-1}
=A_{m-1}\cdots A_0$ 
where $A_j(x)=(Df)(f^jx)^{-1}$.
Hence,  
\begin{align*}
\|Df_X(x)^{-1}- & Df_X(y)^{-1}\| 
=\|A_{m-1}(x)\cdots A_0(x)- A_{m-1}(y)\cdots A_0(y)\|
\\&
\le\sum_{i=0}^{m-1}\|A_{m-1}(x)\cdots A_{i+1}(x)\|
\|A_i(x)-A_i(y)\|
\|A_{i-1}(y)\cdots A_0(y)\|
\\ &
\le\sum_{i=0}^{m-1}\|(Df^{m-i-1})(f^{i+1}x)^{-1}\|
\|A_i(x)-A_i(y)\|
\|(Df^i)(y)^{-1}\|
\\ & \le m^{2q}\sum_{i=0}^{m-1}
\|A_i(x)-A_i(y)\|,
\end{align*}
by~(iii).
By~(iv),
\[
\|A_i(x)-A_i(y)\|=\|Df(f^ix)^{-1}-Df(f^iy)^{-1}\|
\ll \|f^ix-f^iy\|^\eta\ll \|x-y\|^{\eta\delta}m^{\eta q}.
\]
Hence by~(ii),
\begin{align*}
\big|\log\|Df_X(x)^{-1}\|- \log\|Df_X(y)^{-1}\|\big|
& \ll m^{3q+\eta q+1} \|x-y\|^{\eta \delta}.
\\ & \ll \|Df_X(x)^{-1}\|^{-(3q+\eta q+1)/\delta} \|x-y\|^{\eta \delta}.
\end{align*}
This verifies condition~(C2). 

Hence we conclude from~\cite[Theorem~A.1]{AraujoMsub} that for any $q>1$ the map $f_X:X\to X$ is modelled by a Young tower $\DE=Y^\sigma$ with $\mu_Y(\sigma>n)=O(n^{-q})$.

Finally, we note that the construction in~\cite{AFLV11} uses~\cite[Main Theorem~1]{AlvesLuzzattoPinheiro05} where it is made explicit that $Y$ together with its partition elements $a\in\alpha$ are diffeomorphic to open balls in $\R^k$ with the property that $f_X^\sigma$ maps each $a$ diffeomorphically onto $Y$.
In particular, the connected set $f_X^\ell a$ lies in one of the subsets $X_m$ for each $0\le\ell<\sigma(a)-1$, so $h$ is constant on $f_X^\ell a$.  
Hence $f$ possesses a Chernov-Markarian-Zhang structure.
\end{proof}

\begin{rmk} \label{rmk:stretch}
In the situation of Lemma~\ref{lem:NUE},
suppose in addition that condition~(vi) is improved to stretched exponential decay of correlations and that $d\mu_X/d\Leb$ is bounded below. Then part~(2) of~\cite[Theorem~C]{AFLV11} yields Young towers $\DE$ with stretched exponential tails.  In particular, if the rate of decay of correlations is exponential and
$d\mu_X/d\Leb$ is bounded below, then for every $\gamma\in(0,\frac19)$, there exists $Y\subset X$ and $c>0$ such that
$f_X:X\to X$ is modelled by a Young tower $Y^\sigma$ with
$\mu_Y(\sigma>n)=O(e^{-cn^\gamma})$.

A standard argument (see for example~\cite[Theorem~A and p.~1210]{HuVaienti09}) shows that
$d\mu_X/d\Leb$ is bounded below whenever $f$ is noncontracting and topologically exact.
\end{rmk}

\subsection{Upper bounds and limit laws}
\label{sec:limit}

Although the emphasis in this paper is on lower bounds, we obtain essentially optimal upper bounds and many strong statistical properties as a consequence of Lemma~\ref{lem:NUE}.  

Suppose that in the situation of Lemma~\ref{lem:NUE}, $\mu_X(h>n)=O(n^{-\beta})$ for some $\beta>1$.
Then $\mu_Y(\varphi>n)=O(n^{-(\beta-\eps)})$ by Proposition~\ref{prop:rough}(c)
where $\eps>0$ is arbitrarily small.
  Hence by~\cite{Young99}, we have the upper bound
\[
\rho_{v,w}(n)=O(\|v\|_\cH|w|_\infty\, n^{-(\beta-1-\eps)})
\quad\text{for  all $v\in\cH(M)$, $w\in L^\infty(M)$, $n\ge1$.}
\]

By~\cite{DedeckerMerlevede15,GouezelM14,MN08}, 
large deviation estimates and moment bounds follow from this upper bound for all \mbox{$\beta>1$}.  
For $\beta>2$, we obtain the following properties.
 The central limit theorem (CLT) and weak invariance principle (WIP) follow from~\cite{MN05}.  For error rates (Berry-Esseen estimates) in the CLT, and the local CLT, see~\cite{Gouezel05}.   
The almost sure invariance principle with rates follows by~\cite{CunyDedeckerKorepanovMerlevedesub,CunyMerlevede15,Korepanov18}.

Homogenization (convergence of fast-slow systems to a stochastic differential equation) when the fast dynamics is one of these maps $f:M\to M$ follows from~\cite{CFKMZ,GM13,KM16}.  Convergence rates in the WIP and homogenization are obtained in~\cite{AntoniouMapp}.

\subsection{Lower bounds}
\label{sec:cor}

We continue to suppose that we are in the situation of Lemma~\ref{lem:NUE}
and that $\mu_X(h>n)=O(n^{-\beta})$ for some $\beta>1$.  Let $\eps>0$.
Again $\mu_Y(\varphi>n)=O(n^{-(\beta-\eps)})$ and also
$\gamma_n=O(n^{-(\beta-\eps)})$ by Proposition~\ref{prop:int}.
Hence it follows from Theorem~\ref{thm:main}(a) that
\[
\rho_{v,w}(n)=\bvarphi^{-1}d{\textstyle\sum_{j>n/d}\,}\mu_Y(\varphi>jd)\int_M v\,d\mu\int_M w\,d\mu+O(n^{-(s-\eps)})
\]
for all $v\in \cH(X)$, $w\in L^\infty(X)$,
where $s=\min\{2(\beta-1),\beta\}$.
By Theorem~\ref{thm:main})(b),
$\rho_{v,w}(n)=O(n^{-(\beta-\eps)})$ for all
$v\in\cH(X)$ with $\int_Mv\,d\mu=0$ and all $w\in L^\infty(M)$.

If moreover, $\mu_X(h>n)\approx n^{-\beta}$, then by Proposition~\ref{prop:rough}(c,d), 
\[
n^{-(\beta-1+\eps)}\ll \rho_{v,w}(n) \ll n^{-(\beta-1-\eps)},
\]
for all $v\in \cH(X)$, $w\in L^\infty(X)$ with nonzero mean.

\subsection{Application to Hu-Vaienti maps}
\label{sec:HV}
We continue to consider local diffeomorphisms $f:M\to M$, where $M\subset\R^k$ is compact, with finitely many branches as in Subsection~\ref{sec:piecewise}.
We now specialize to intermittent maps with a neutral fixed point at $0$ as described in Section~\ref{sec:HuV}.
These maps are piecewise $C^{1+\eta}$ for some $\eta\in(0,1)$ with finitely many branches, noncontracting everywhere (so $\|Df(x)v\|\ge \|v\|$ for all $x\in \R^k$, $v\in \R^k$), and expanding everywhere except at $0$
(so $\|Df(x)v\|>\|v\|$ for all $v\in\R^k$ if and only if $x\neq 0$).

The existence of absolutely continuous invariant probability measures for one-dimensional intermittent maps was studied 
 by~\cite{Thaler80} when the maps are Markov and by~\cite{Zweimuller98} in the nonMarkov case.  
In~\cite{HuVaienti09},  a Banach space of quasi-H\"older observables studied by~\cite{Keller85,Saussol00} was used to establish existence of $\sigma$-finite absolutely continuous ergodic $f$-invariant measures $\mu$ on $M$ 
for multidimensional nonMarkov nonuniformly expanding maps.  
The cases $\mu(M)<\infty$ and $\mu(M)=\infty$ are considered equally in~\cite{HuVaienti09}; here we focus on the case of finite measures.
The results in~\cite{HuVaienti09} require a delicate analysis taking into account poor distortion properties of multidimensional nonuniformly expanding maps.
In~\cite{HuVaientiapp}, the quasi-H\"older space was used further to analyze upper and lower bounds on decay of correlations.  
Here we show how to combine~\cite{HuVaienti09} and  Lemma~\ref{lem:NUE} to obtain the essentially optimal results mentioned in Section~\ref{sec:HuV}.

To fix ideas, we focus on~\cite[Example~5.1]{HuVaientiapp}, setting
\[
f(x)=x(1+|x|^\gamma+O(|x|^{\gamma'}))
\]
 for $x$ close to $0$ where
$\gamma\in(0,k)$ and $\gamma'>\gamma$.
Recall that the domains of the branches are denoted $U_1,\dots,U_K$ and have piecewise smooth boundaries; we assume that $0\in\Int U_1$ and $f^{-1}0\cap \bigcup\partial U_j=\emptyset$.  
This means that Assumptions~1 and~2 of~\cite[Theorem~A]{HuVaienti09} are satisfied.
Also, we assume that $f:M\to M$ is topologically exact.
Our final assumption is a growth of complexity condition,
Assumption~3 in~\cite[Theorem~A]{HuVaienti09},
which is too technical to reproduce here.  
As pointed out in~\cite[Remark~5.2]{HuVaientiapp} it follows from~\cite[Lemma~2.1]{Saussol00} that we can arrange for Assumption~3 to be satisfied by choosing $f$ to be sufficiently expanding outside of a suitable neighborhood of $0$.

Choose an open ball $R$ with $0\in R\subset U_1$ such that
$\bar R\subset fR$ and $\overline{fR}\subset U_1$.
Set $X=M\setminus R$.

\begin{prop} \label{prop:HuV}
There is a unique absolutely continuous invariant probability measure $\mu_X$ on $X$.
The assumptions of Lemma~\ref{lem:NUE} are satisfied and 
$\mu_X(h>n)\approx  n^{-\beta}$ where $\beta=k/\gamma$.
\end{prop}

\begin{proof}
The singular set $\partial S$ is a countable union of piecewise smooth submanifolds limiting on finitely many piecewise smooth submanifolds, so 
condition~(i) is satisfied.

Since the first return set $X$ is bounded away from $0$, it is immediate from noncontractivity on $M$ and uniform expansion on $X$ that
$\|(Df_X)^{-1}\|\le\lambda<1$.  
The remaining estimates in~(ii) are established in~\cite{HuVaienti09,HuVaientiapp}.  (A big advantage here is that $\delta$ can be taken arbitrarily small and $q$ arbitrarily large, so the fine details in~\cite{HuVaienti09,HuVaientiapp} such as unbounded distortion are not an issue.)
Since $f$ is noncontracting, conditions~(iii) and~(iv) hold by Remark~\ref{rmk:NUE}.  

A key step in~\cite{HuVaienti09} is to establish quasicompactness of the transfer operator for the first return map $f_X:X\to X$.
Assumptions~1--3 of~\cite[Theorem~A]{HuVaienti09} are mentioned explicitly above.  
As noted in~\cite[Example~5.1]{HuVaientiapp}, Assumption 4 is automatic.  
Hence~\cite[Theorem~A]{HuVaienti09} guarantees the existence of 
an absolutely continuous invariant probability measure $\mu_X$ on $X$.  The density is quasi-H\"older and hence 
 lies in $L^\infty(X)$ verifying condition~(v) of Lemma~\ref{lem:NUE}.
By Remark~\ref{rmk:stretch}, the density is also bounded below and hence $\mu_X$ is unique.

Moreover,~\cite[Theorem~A]{HuVaienti09} establishes quasicompactness in the quasi-H\"older space and hence $\mu_X$ is mixing up to a finite cycle.  Since the support of a nonvanishing quasi-H\"older function has nonempty interior~\cite[Lemma~3.1]{Saussol00}, it follows from topological exactness that $\mu_X$ is mixing.
Condition~(vi) is now an immediate consequence of quasicompactness.

By~\cite{HuVaienti09}, 
each $X_m$ is a finite union of approximately spherical shells bounded by hypersurfaces $S_m$ and $S_{m+1}$ where $S_m$ is approximately a sphere of radius $\approx m^{-1/\gamma}$.  
It follows that $\mu_X(X_m)\ll\Leb(X_m)\ll m^{-k/\gamma}$ so condition~(vii) is satisfied.

By Remark~\ref{rmk:stretch}, topological exactness ensures that
$\mu_X(h>n)\approx\Leb(h>n)$.
Moreover, $\{h>n\}=\bigcup_{m>n}X_m$ is a finite union of balls of radius
$\approx n^{-1/\gamma}$ so
$\mu_X(h>n)\approx n^{-\beta}$.
\end{proof}

Hence for $\gamma\in(0,k)$, we can apply the results in Subsections~\ref{sec:limit} and~\ref{sec:cor} to obtain the upper and lower bounds in~\eqref{eq:HV}, as well as the limit laws mentioned in Section~\ref{sec:limit}.

\section{Two-sided version of the main result}
\label{sec:main2}

In this section, we extend Theorem~\ref{thm:main} to invertible maps.
A two-sided analogue of the Chernov-Markarian-Zhang structure is 
described in Subsection~\ref{sec:prel2}.
The main result of this section, Theorem~\ref{thm:main2}, is stated in Subsection~\ref{sec:state2} and reformulated for towers in Subsection~\ref{sec:Delta2}.
In Subsection~\ref{sec:approx}, we show how to approximate two-sided observables by one-sided observables.
In Subsection~\ref{sec:proof2}, we complete the proofs.

\subsection{Preliminaries}
\label{sec:prel2}

We describe a two-sided (invertible) analogue of the structures discussed in Section~\ref{sec:prel}.  Throughout, $f:M\to M$, $f_X:X\to X$, $f_\Delta:\Delta\to\Delta$ and $F:Y\to Y$ are all two-sided versions of the maps from Section~\ref{sec:prel}, and the one-sided versions are denoted $\bF:\bY\to\bY$ and so on.   We continue to write 
$\bvarphi=\int_Y\varphi\,d\mu_Y$ but as will become clear this does not cause any confusion.

\paragraph{Two-sided Gibbs-Markov maps}

Let $(Y,d)$ be a bounded metric space with Borel probability measure $\mu_Y$
and let $F:Y\to Y$ be an ergodic measure-preserving transformation.
Let $\bF:\bY\to\bY$ be a full-branch Gibbs-Markov map
with partition $\alpha$ and ergodic invariant probability measure $\bar\mu_Y$.

We suppose that there is a measure-preserving semiconjugacy $\bar\pi:Y\to\bY$,
so $\bar\pi\circ F=\bF\circ\bar\pi$ 
and $\bar\pi_*\mu_Y=\bar\mu_Y$.
The separation time on $\bY$ lifts to a separation time on $Y$ given by
$s(y,y')=s(\bar\pi y,\bar\pi y')$ for $y,y'\in Y$.
Suppose that 
there exist constants $C>0$, $\theta\in(0,1)$ such that
\begin{equation} \label{eq:prod}
d(F^ny,F^ny')\le C(\theta^n+\theta^{s(y,y')-n})
\quad\text{for all $y,y'\in Y$, $n\ge1$}.
\end{equation}
Then we call $F:Y\to Y$ a {\em two-sided Gibbs-Markov map}.

\paragraph{Two-sided Young towers}
Let $F:Y\to Y$ be a two-sided Gibbs-Markov map on $(Y,\mu_Y)$ and let $\varphi:Y\to\Z^+$ be an integrable function that is constant on $\bar\pi^{-1}a$ for each $a\in\alpha$.
In particular, such a return time $\varphi$ is well-defined on $\bY$.  
Define the one-sided Young tower $\bar\Delta=\bY^\varphi$ and tower map $\bar f_\Delta:\bar\Delta\to\bar\Delta$ as in Section~\ref{sec:prel}.
Using $F:Y\to Y$ instead of $\bF:\bY\to\bY$, we also define
the {\em two-sided Young tower} $\Delta=Y^\varphi$ and {\em tower map} $f_\Delta:\Delta\to\Delta$.
We obtain ergodic invariant probability measures 
$\mu_\Delta=(\mu_Y\times{\rm counting})/\bvarphi$ 
and $\bar\mu_\Delta=(\bar\mu_Y\times{\rm counting})/\bvarphi$ 
on $\Delta$ and $\bar\Delta$.

The projection $\bar\pi:Y\to\bY$ extends to $\bar\pi:\Delta\to\bar\Delta$ with
$\bar\pi(y,\ell)=(\bar\pi y,\ell)$.  This defines a measure-preserving semiconjugacy between $f_\Delta$ and $\bar f_\Delta$.

Now suppose that 
$f:M\to M$ is an ergodic measure-preserving transformation
on a probability space $(M,\mu)$, 
and that $Y\subset M$ is measurable with $\mu(Y)>0$.
Suppose that  $F:Y\to Y$ 
is a two-sided Gibbs-Markov map with respect to a probability $\mu_Y$ on $Y$, and that $\varphi:Y\to\Z^+$ is a return time as above.
Form the tower $\Delta=Y^\varphi$ and tower map $f_\Delta:\Delta\to\Delta$.  The map $\pi_M:\Delta\to M$, $\pi_M(y,\ell)=f^\ell y$ defines a semiconjugacy between $f_\Delta$ and $f$.  We require moreover that $(\pi_M)_*\mu_\Delta=\mu$.
Then we say that $f$ is {\em modelled by a two-sided Young tower}.

\paragraph{Two-sided Chernov-Markarian-Zhang structure}

Let $(M,d)$ be a bounded metric space with Borel probability measure $\mu$ and let
$f:M\to M$ be an ergodic and mixing measure-preserving transformation.
Suppose that $Y\subset X\subset M$ are Borel sets with $\mu(Y)>0$.
Define
the first return time $h:X\to\Z^+$ and first return map $f_X=f^h:X\to X$. 

We require that $f_X:X\to X$ is modelled by a two-sided Young tower $\DE=Y^\sigma$ with 
return time $\sigma:Y\to\Z^+$ and return map $F=f_X^\sigma:Y\to Y$.
Here, $F=f_X^\sigma:Y\to Y$ is a two-sided Gibbs-Markov map with ergodic invariant probability measure $\mu_Y$ and partition $\alpha$ such that
$\sigma$ is constant on partition elements.  We require in addition that $h$ is constant on $f_X^\ell \bar\pi^{-1}a$
for all $a\in\alpha$, $0\le\ell\le\sigma(a)-1$.

Define the {\em induced return time}
$\varphi=h_\sigma:Y\to\Z^+$ as in~\eqref{eq:induce}.
Then $\varphi$ is an integrable return time (constant on $\bar\pi^{-1}a$ for $a\in\alpha$).
In particular, $f:M\to M$ is modelled by a Young tower $\Delta=Y^\varphi$ with the same two-sided Gibbs-Markov map $F=f_X^\sigma=f^\varphi$.

We say that $f:M\to M$ satisfying these assumptions possesses a {\em two-sided Chernov-Markarian-Zhang structure}.

\begin{rmk} \label{rmk:CMZ2}
Young~\cite{Young98} introduced Young towers with exponential tails as a general method for dealing with diffeomorphisms with singularities; the initial landmark application was to prove exponential decay of correlations for planar finite horizon dispersing billiards.  Chernov~\cite{Chernov99}  simplified the construction of exponential Young towers and used this to prove exponential decay of correlations for planar dispersing billiards with infinite horizon.
Then Young~\cite{Young99} studied examples with subexponential decay of correlations using Young towers with subexponential tails.  Markarian~\cite{Markarian04}, noting that Chernov's simplification no longer applies in the subexponential case, devised the method outlined in this section: namely to construct a first return map for which Chernov~\cite{Chernov99} applies.  This was used to prove the decay of correlations bound $O(1/n)$ for Bunimovich stadia.  The method was extended and simplified by
Chernov \& Zhang~\cite{ChernovZhang05}
who applied it to a large class of billiard examples.  Subsequent applications of the method include~\cite{ChernovMarkarian07,ChernovZhang05b,Zhang17a}.
\end{rmk}

\begin{rmk}  We have omitted much of the structure often associated with Young towers, mentioning only those properties required in the sequel.
For instance, we have not made any explicit mention of a product structure, though we make use of 
condition~\eqref{eq:prod} which is a consequence.
Similarly, we have not made explicit the quotienting procedure (along local stable leaves) that passes from $F$ to $\bF$.
\end{rmk}

\paragraph{Two-sided dynamically H\"older observables}
Suppose that $f:M\to M$ admits a two-sided Chernov-Markarian-Zhang structure as above.
Fix $\theta\in(0,1)$.  For $v:M\to\R$, define
\[
{\|v\|}_\cH=|v|_\infty+{|v|}_\cH, \quad
{|v|}_\cH=\sup_{y,y'\in Y,\,y\neq y'}\sup_{0\le\ell<\varphi(y)}
\frac{|v(f^\ell y)-v(f^\ell y')|}{d(y,y')+\theta^{s(y,y')}}.
\]
We say that $v$ is {\em dynamically H\"older}
if ${\|v\|}_\cH<\infty$ and denote by $\cH(M)$ the space of such observables.
Write $\cH(X)=\{v\in \cH(M):\supp v\subset X\}$.

Again, it is standard that H\"older observables are dynamically H\"older for the classes of dynamical systems of interest in this paper:

\begin{prop} \label{prop:dyn2}
Let $\eta\in(0,1]$ and let $d_0$ be a bounded metric on $M$.  
Let $C^\eta(M)$ be the space of observables that are $\eta$-H\"older with respect to $d_0$.  Suppose that
there exist $K>0$, $\gamma_0\in(0,1)$ such that
$d_0(f^\ell y,f^\ell y')\le K(d_0(y,y')+\gamma_0^{s(y,y')})$
 for all $y,y'\in Y$, $0\le\ell\le \varphi(y)-1$.

Then 
$C^\eta(M)\subset \cH(M)$ where we may choose any $\theta\in[\gamma_0^\eta,1)$
and $d=d_0^{\eta'}$ for any $\eta'\in(0,\eta]$.
\end{prop}

\begin{proof}
Let $v\in C^\eta(M)$, $y,y'\in Y$, $0\le\ell<\varphi(y)-1$. Then
\begin{align*}
|v(f^\ell y)- v(f^\ell y')|
 \le {|v|}_{C^\eta}d_0(f^\ell y,f^\ell y')^\eta
   & \le K^\eta {|v|}_{C^\eta} (d_0(y,y')^\eta+\gamma_0^{\eta s(y,y')})
   \\ & \ll {|v|}_{C^\eta} (d(y,y')+\theta^{s(y,y')}).
\end{align*}
Hence ${|v|}_{\cH}\ll {|v|}_{C^\eta}$
and it follows that $v\in \cH(M)$.
\end{proof}

\subsection{Statement of the main result}
\label{sec:state2}

As in Section~\ref{sec:main}, we provide an abstract result for maps $f:M\to M$ with a Chernov-Markarian-Zhang structure under the assumption that $\mu_Y(\varphi>n)=O(n^{-\beta'})$ for some $\beta'>1$.   In this situation it follows from Young~\cite{Young99} that $\rho_{v,w}(n)=O(n^{-(\beta'-1)})$ for dynamically H\"older observables.
(The result in~\cite{Young99} is formulated for one-sided systems; see~\cite[Theorem~2.10]{KKM19} or~\cite[Appendix~B]{MT14} for the two-sided case.)
We obtain a lower bound for dynamically H\"older observables supported in $X$.

Define $\sigma_n,\,\gamma_n$ as in~\eqref{eq:gamma}
and $\zeta_{\beta'}$ as in~\eqref{eq:zeta}.

\begin{thm} \label{thm:main2}
Let $f:M\to M$ be a map with a two-sided Chernov-Markarian-Zhang structure, and suppose that $\mu_Y(\varphi>n)=O(n^{-\beta'})$ for some $\beta'>1$. 
Then there is a constant $C>0$ such that for all $n\ge1$,
\begin{itemize}
\item[(a)]
$\displaystyle\Big|\rho_{v,w}(n)-\bvarphi^{-1}\sum_{j>n} \mu_Y(\varphi>j)\int_Mv\,d\mu\int_Mw\,d\mu\Big|
\le C{\|v\|}_{\cH}{\|w\|}_{\cH}(\gamma_{[n/3]}+\zeta_{\beta'}(n))$
for all $v,\,w\in \cH(X)$,
\item[(b)]
$|\rho_{v,w}(n)|\le C{\|v\|}_{\cH}{\|w\|}_\cH\gamma_{[n/3]}$ 
for all $v,\,w\in \cH(X)$ with $\int_M v\,d\mu=0$.
\end{itemize}
\end{thm}

\begin{rmk} \label{rmk:NUH}
The classical Smale-Williams solenoid construction can be adapted (see for example~\cite[Section~5]{AlvesPinheiro08} and~\cite[Example~4.2]{MV16}) to construct intermittent maps that are the invertible analogue of the Hu-Vaienti maps in Section~\ref{sec:HV}.  (The constructions in~\cite{AlvesPinheiro08,MV16} are written down for one-dimensional maps but apply equally to multidimensional maps.)  
The resulting solenoidal intermittent maps fall within the two-sided Chernov-Markarian-Zhang framework and have stable and unstable directions of any specified dimension.  Our results yield essentially optimal upper and lower bounds on decay of correlations for these examples.  
Again, the lower bounds are realized by H\"older observables that are supported away from the neutral fixed point.  
\end{rmk}

\subsection{Tower reformulation}
\label{sec:Delta2}

Let $d=\gcd\{\varphi(a):a\in\alpha\}$.  
As in Section~\ref{sec:Delta}, we replace the tower $\Delta=Y^\varphi$ by a mixing tower $\Delta=Y^\Phi$ where $\Phi=d^{-1}\varphi$.  Again we have a measure-preserving semiconjugacy $\pi_M(y,\ell)=g^\ell y$ between 
$(\Delta,f_\Delta,\mu_\Delta)$ and $(M,g,\mu)$ where $g=f^d$.
We consider observables $v:M\to\R$ supported on $X_d=X\cup\dots\cup f^{-(d-1)}X$ and the corresponding lifted observables $\tv=v\circ\pi_M:\Delta\to\R$ supported in $\hX=\pi_M^{-1}X$.

Fix $\theta\in(0,1)$ and define
\begin{equation} \label{eq:norm}
\|\tv\|_\theta=|\tv|_\infty+|\tv|_\theta, \quad 
|\tv|_\theta
 =\sup_{y,y'\in Y,\,y\neq y'}\sup_{0\le\ell\le \Phi(y)-1}
\frac{|\tv(y,\ell)-\tv(y',\ell)|}{d(y,y')+\theta^{s(y,y')}}.
\end{equation}
Let $\cF_\theta(\hX)$ denote the space of observables $\tv$ supported in $\hX$ with $\|\tv\|_\theta<\infty$.

Given $\tv,\tw\in L^\infty(\Delta)$, define
\[
\rho^*_{\tv,\tw}(n)=\int_\Delta \tv\,\tw\circ f_\Delta^n\,d\mu_\Delta.
\]

The counterpart of Theorem~\ref{thm:Delta} is:

\begin{thm} \label{thm:Delta2}
There  is a constant $C>0$ such that for all $n\ge1$,
\begin{itemize}
\item[(a)]
$\Big|\rho^*_{\tv,\tw}(n)-(1+\bPhi^{-1}\sum_{j>n} \mu_Y(\Phi>j))\int_\Delta\tv\,d\mu_\Delta\int_\Delta\tw\,d\mu_\Delta\Big|
\le C\|\tv\|_\theta\|\tw\|_\theta(\gamma_{[n/2]d}+\zeta_{\beta'}(n))$
for all $\tv,\tw\in \cF_\theta(\hX)$, 
\item[(b)]
$|\rho^*_{\tv,\tw}(n)|\le C\|\tv\|_\theta\|\tw\|_\theta\,\gamma_{[n/2]d}$ for all
$\tv,\tw\in \cF_\theta(\hX)$ with $\int_\Delta \tv\,d\mu_\Delta=0$.
\end{itemize}
\end{thm}

Theorem~\ref{thm:main2} is a direct consequence of Theorem~\ref{thm:Delta2} in exactly the same way that
Theorem~\ref{thm:main} was a direct consequence of Theorem~\ref{thm:Delta}.
Hence we omit the details except to mention that we make use of the estimate
$\gamma_{[n/(2d)]d}\ll \gamma_{[n/3]}$.

The key steps in the proof of Theorem~\ref{thm:Delta2} are contained in the following result.

\begin{lemma}
\label{lem:Delta2}
There is a constant $C>0$ such that for all $n\ge1$:
 \begin{itemize}
 \item[(a)]  
 $|\rho^*_{1_\hX,\tw}(n)|\le C\|\tw\|_\theta(\sigma_{[n/2]d}+\zeta_{\beta'}(n))$ for all
 $\tw\in \cF_\theta(\hX)$ with $\int_\Delta \tw\,d\mu_\Delta=0$,
 \item[(b)] 
$|\rho^*_{\tv,\tw}(n)|
\le C\|\tv\|_\theta\|\tw\|_\theta\,\gamma_{[n/2]d}
$
for all $\tv,\tw\in\cF_\theta(\hX)$ with $\int_\Delta \tv\,d\mu_\Delta=0$.
 \end{itemize}
\end{lemma}

The proof of Lemma~\ref{lem:Delta2} takes up most of the remainder of this section.
Assuming this result, we can complete the proof of Theorem~\ref{thm:Delta2}.

\begin{pfof}{Theorem~\ref{thm:Delta2}}
Let $a=\mu_\Delta(\hX)^{-1}\int_\Delta \tv\,d\mu_\Delta$, $b=\mu_\Delta(\hX)^{-1}\int_\Delta \tw\,d\mu_\Delta$ and define
$v_0=\tv-a1_\hX$ and $w_0=\tw-b1_\hX$.  Then $\int_\Delta v_0\,d\mu_\Delta=\int_\Delta w_0\,d\mu_\Delta=0$ and
\[
\rho^*_{\tv,\tw}=\rho^*_{a1_\hX,b1_\hX}+\rho^*_{a1_\hX,w_0}+ \rho^*_{v_0,\tw}.
\]
Note that
\[
\rho_{a1_\hX,b1_\hX}^*(n)=
\int_\Delta a1_\hX\, b1_\hX\circ f_\Delta^n \,d\mu_\Delta=
\int_\bDelta a1_\bZ\, b1_\bZ\circ \bar f_\Delta^n \,d\bar\mu_\Delta
\]
so by Theorem~\ref{thm:Delta}(a),
\[
\Big|\rho^*_{a1_\hX,b1_\hX}(n)-\Big(1+\bPhi^{-1}\sum_{j>n} \mu_Y(\Phi>j)\Big)\int_\Delta \tv\,d\mu_\Delta \int_\Delta \tw\,d\mu_\Delta\Big|
 \ll |\tv|_\infty|\tw|_\infty(\gamma_{nd}+\zeta_{\beta'}(n)).
\]
By Lemma~\ref{lem:Delta2},
$|\rho^*_{a1_\hX,w_0}(n)|\ll
 |\tv|_\infty\|\tw\|_\theta(\sigma_{[n/2]d}+\zeta_{\beta'}(n))$ and
$|\rho^*_{v_0,\tw}(n)|\ll \|\tv\|_\theta \|\tw\|_\theta\gamma_{[n/2]d}$.
Part (a) follows from these combined estimates.

If in addition, $\int_\Delta \tv\,d\mu_\Delta=0$, then 
$\rho^*_{\tv,\tw}(n) =\rho^*_{v_0,\tw}(n)$ yielding part~(b). 
\end{pfof}

\begin{prop} \label{prop:hXbar}
Let $p,p'\in \Delta$ with $\bar\pi p=\bar\pi p'$.
Then $p\in\hX$ if and only if $p'\in\hX$.
\end{prop}

\begin{proof}
Write $p=(y,\ell)$, $p'=(y',\ell)$ where $\bar \pi y=\bar\pi y'$ and
 $\ell\in\{0,\dots,\Phi(y)-1\}$.
Since $h$ is the first return time under $f$ to $X$,  we have $g^\ell y\in X_d$ if and only if 
$\ell d \le \sum_{j=0}^{k-1}h(f_X^jy)<(\ell+1)d$ 
for some $k=0,\dots,\sigma(a)-1$.  Now use that $h$ is constant on $f_X^j\bar\pi^{-1}(\bar\pi y)$.
\end{proof}

\noindent By Proposition~\ref{prop:hXbar}, we can write
$\hX=\bar\pi^{-1}\bZ$ where $\bZ=\bar\pi\hX\subset\bDelta$.
By Proposition~\ref{prop:hX},
\begin{equation} \label{eq:sigma}
\sum_{\ell=0}^{\Phi(\bar y)-1}1_\bZ(\bar y,\ell)
\le \sigma(\bar y)
\quad\text{for all $\bar y\in\bY$.}
\end{equation}

\subsection{Approximation by one-sided observables}
\label{sec:approx}

In this subsection, we show how to approximate two-sided observables by one-sided observables,
broadly following the method used in~\cite[Appendix~B]{MT14} which was in turn based on a private communication by Gou\"ezel.  Using this we prove
Lemma~\ref{lem:Delta2}(b).

Extend the separation time $s$ on $Y$ to $\Delta$ by setting $s((y,\ell),(y',\ell'))=s(y,y')$ when $\ell=\ell'$ and $0$ otherwise.
Let $\psi_n=\sum_{j=0}^{n-1}1_Y\circ f_\Delta^j$ be the number of entries to $Y$.
For $\tv\in L^\infty(\Delta)$, 
we approximate $\tv\circ f_\Delta^n$ by 
\[
\tv_n:\Delta\to\R, \qquad 
\tv_n(p)=\inf\{\tv\circ f_\Delta^n(q):s(p,q)\ge 2\psi_n(p)\}, \quad n\ge1.
\]

Let $L:L^1(\bDelta)\to L^1(\bDelta)$ denote the transfer operator for $\bar f_\Delta$.
\begin{prop} \label{prop:tv}
The function $\tv_n$ lies in $L^\infty(\Delta)$ and projects down to an
observable $\bar v_n\in L^\infty(\bar\Delta)$.
Moreover, there exists $C>0$ such that 
for all $\tv\in\cH(\Delta)$, $n\ge1$,
\begin{itemize}
\item[(a)] $|\bar v_n|_\infty=|\tv_n|_\infty\le |\tv|_\infty$.
\item[(b)]
If $\supp \tv\subset \hX$, then $\supp \tv_n\subset f_\Delta^{-n}\hX$
and 
$\supp L^n\bar v_n\subset \bZ$.
\item[(c)]
$|\tv\circ f_\Delta^n-\tv_n|\le C|\tv|_\theta\, \theta^{\psi_n}$.
\item[(d)] $|(L^n\bar v_n)(\bar p_1)-(L^n\bar v_n)(\bar p_2)|\le C\|\tv\|_\theta\,\theta^{s(\bar p_1,\bar p_2)}$ for all $\bar p_1,\bar p_2\in\bDelta$.
\end{itemize}
\end{prop}

\begin{proof}
If $s(p,p')\ge 2\psi_n(p)$, then $\tv_n(p)=\tv_n(p')$.
It follows that $\tv_n$ is piecewise constant on a measurable partition
of $\Delta$, and hence is measurable, and that $\bar v_n$ is well-defined.
Parts (a) and (b) are immediate.

Let $p=(y,\ell)\in\Delta$.  Then
\[
|\tv\circ f_\Delta^n(p)-\tv_n(p)| =|\tv\circ f_\Delta^n(p)-\tv\circ f_\Delta^n(q)|
\]
where $q=(z,\ell)$ is such that $s(p,q)\ge 2\psi_n(p)$.
Now,
$f_\Delta^np=(F^{\psi_n(p)}y,\ell_1)$ where $\ell_1=\ell+n-\Phi_{\psi_n(p)}(p)$
and similarly 
$f_\Delta^nq=(F^{\psi_n(p)}z,\ell_1)$.
(Here, $\Psi_k=\sum_{j=0}^{k-1}\Psi\circ F^j$.)
By definition of $|\tv|_\theta$ and~\eqref{eq:prod},
\begin{align*}
|\tv\circ f_\Delta^n(p)-\tv_n(p)| & =
|\tv(F^{\psi_n(p)}y,\ell_1) -\tv(F^{\psi_n(p)}z,\ell_1)|
\\ & \le |\tv|_\theta(d(F^{\psi_n(p)}y,F^{\psi_n(p)}z)+\theta^{s(F^{\psi_n(p)}y,F^{\psi_n(p)}z)})
\\ & \ll |\tv|_\theta (\theta^{\psi_n(p)}+\theta^{s(y,z)-\psi_n(p)})
\ll |\tv|_\theta \theta^{\psi_n(p)}.
\end{align*}
This proves part~(c).

To prove (d), recall that
$(L^n\bar v_n)(\bar p)=\sum_{\bar f_\Delta^n\bar q=\bar p}\xi_n(\bar q)\bar v_n(\bar q)$
where $\xi$ is the weight function.
Write
\begin{align*}  
(L^n\bar v_n)(\bar p_1)- (L^n\bar v_n)(\bar p_2) & = I_1+I_2
\end{align*}
where
\[
I_1  = \sum_{\bar f_\Delta^n\bar q_1=\bar p_1}(\xi_n(\bar q_1)-\xi_n(\bar q_2))\bar v_n(\bar q_2), \qquad
I_2  = \sum_{\bar f_\Delta^n\bar q_1=\bar p_1}\xi_n(\bar q_1)(\bar v_n(\bar q_1)-\bar v_n(\bar q_2)).
\]
As usual, we pair up preimages so that
\begin{equation} \label{eq:pairs}
s(\bar q_1,\bar q_2)=\psi_n(\bar q_1)+s(\bar p_1,\bar p_2).
\end{equation}

A standard argument shows that 
$|\xi_n(\bar q_1)-\xi_n(\bar q_2)|\ll \xi_n(\bar q_1)\theta^{s(\bar p_1,\bar p_2)}$.
Hence, 
\[
|I_1|\ll |\tv|_\infty\, \theta^{s(\bar p_1,\bar p_2)}.
\]

Next, choose $q_j\in\Delta$ that project to $\bar q_j$ and write
\[
\bar v_n(\bar q_1)-\bar v_n(\bar q_2)=
\tv\circ f_\Delta^n(\hat q_1 )-\tv\circ f_\Delta^n(\hat q_2),
\]
where
$\hat q_1,\hat q_2\in\Delta$ satisfy
\begin{equation} \label{eq:qhat}
s(\hat q_j,q_j)\ge 2\psi_n(\bar q_j) =2\psi_n(\hat q_j).
\end{equation}
As in part (c),
\begin{equation} \label{eq:min}
|\tv\circ f_\Delta^n(\hat q_1)-\tv\circ f_\Delta^n(\hat q_2)| 
\ll |\tv|_\theta
(\theta^{\psi_n(\hat q_1)}
+\theta^{s(\hat q_1,\hat q_2)-\psi_n(\hat q_1)}).
\end{equation}
Since $\bar v_n(\bar q_1)=\bar v_n(\bar q_2)$ if $s(q_1,q_2)\ge 2\psi_n(\bar q_1)$, we may suppose without loss that
\[
s(q_1,q_2)\le 2\psi_n(\bar q_1)=2\psi_n(\hat q_1).
\]
Then it follows from~\eqref{eq:pairs} that
\begin{equation} \label{eq:wlog}
\psi_n(\hat q_1)\ge s(q_1,q_2)-\psi_n(\bar q_1) 
=s(\bar p_1,\bar p_2).
\end{equation}
By~\eqref{eq:pairs},~\eqref{eq:qhat} and~\eqref{eq:wlog},
\begin{align*}
s(\hat q_1,\hat q_2)
  \ge \min\{s(q_1,q_2),s(\hat q_1,q_1),s(\hat q_2,q_2)\}
& \ge \min\{\psi_n(\bar q_1)+s(\bar p_1,\bar p_2),2\psi_n(\bar q_1)\}
\\ & \ge \psi_n(\bar q_1)+s(\bar p_1,\bar p_2).
\end{align*}
Substituting this and~\eqref{eq:wlog} into~\eqref{eq:min},
$|\bar v_n(\bar q_1)-\bar v_n(\bar q_2)|
\ll |\tv|_\theta \theta^{s(\bar p_1,\bar p_2)}$.
Hence 
\[
|I_2|\ll |\tv|_\theta\, \theta^{s(\bar p_1,\bar p_2)}
\]
completing the proof.
\end{proof}

We continue to suppose that $\mu_Y(\varphi>n)=O(n^{-\beta'})$ for some $\beta'>1$ with
$\sigma_n$ defined as in~\eqref{eq:gamma}.
Define the operators $R(n):L^\infty(\bY)\to L^\infty(\bY)$, 
\[
R(n)u=L^n(1_{\{\Phi=n\}}u)=R(1_{\{\Phi=n\}}u),\;n\ge1.
\]
Using~\eqref{eq:GM},
a standard calculation~\cite{Gouezel04a,Sarig02} yields:

\begin{prop} \label{prop:GS}
There exists $C>0$ such that
\[
|R(n)|_{L^\infty(\bY)} \le C \mu_Y(\Phi=n)
\]
and
\[
|(R(n)1_\bY)(y)- (R(n)1_\bY)(y')| \le C \mu_Y(\Phi=n) \theta^{s(y,y')}
\]
for all $y,y'\in \bY$, $n\ge1$.  \qed
\end{prop}

\begin{lemma} \label{lem:Ltheta}
There exists $C>0$ such that
$|\int_\bDelta \theta^{\psi_n}1_\bZ\circ \bar f_\Delta^n\,d\bar\mu_\Delta|\le C \sigma_{nd}$ and
$|\int_\bDelta 1_\bY \theta^{\psi_n}\,d\bar\mu_\Delta|\le C n^{-\beta'}$ for all $n\ge1$.
\end{lemma}

\begin{proof}
Define
$L_\theta:L^1(\bDelta)\to L^1(\bDelta)$ by $L_\theta u=L(\theta^{1_Y} u)$.
Then $L_\theta^n u=L^n(\theta^{\psi_n} u)$ and
\begin{align*}
|\theta^{\psi_n}\,1_\bZ\circ f_\Delta^n|_1
=\int_\bDelta L^n\theta^{\psi_n}\,1_\bZ\,d\bar\mu_\Delta
=\int_\bDelta 1_\bZ L_\theta^n1\,d\bar\mu_\Delta.
\end{align*}
Similarly, 
$|1_Y\theta^{\psi_n}|_1
=\int_\bDelta L_\theta^n1_\bY\,d\bar\mu_\Delta$.
Hence it suffices to show that
$|1_\bZ L_\theta^n1|_1\ll \sigma_{nd}$ and
$|L_\theta^n1_\bY|_1\ll n^{-\beta'}$.

In analogy with Section~\ref{sec:TR}, we define
the renewal operators $R_\theta(n),\,T_\theta(n):L^\infty(\bY)\to L^\infty(\bY)$, 
\[
R_\theta(n)u=L_\theta^n(1_{\{\Phi=n\}}u)=\theta R(n)u,\;n\ge1,\qquad
T_\theta(n)u=1_Y L_\theta^n(1_Y u),\,n\ge0,
\]
and the corresponding Fourier series
$\hR_\theta(z),\,\hT_\theta(z):L^\infty(\bY)\to L^\infty(\bY)$, for $z\in\D$,
\[
\hR_\theta(z)=\theta\sum_{n=1}^\infty R(n)z^n,\qquad
\hT_\theta(z)=\sum_{n=0}^\infty T_\theta(n)z^n.
\]
Again, $\hT_\theta=(I-\hR_\theta)^{-1}$ on $\D$.
Moreover, the spectral radius of $\hR_\theta(z)$
on $L^\infty(\bY)$ is at most $\theta$ for $z\in\barD$,
 so $I-\hR_\theta(z)$ is invertible as an operator on $L^\infty(\bY)$ for all $z\in\barD$.
By Proposition~\ref{prop:GS},
${|R(n)|}_{L^\infty(\bY)}\ll n^{-\beta'}$.
It follows from~\cite[Theorem~A.3]{Gouezel05} that
${|T_\theta(n)|}_{L^\infty(\bY)}\ll n^{-\beta'}$.

Next, as in~\cite[Eq.~(11)]{Gouezel05}, we have the decomposition\footnote{The notation here, which is chosen to mimic that in~\cite{Gouezel05}, is local to the proof of this Lemma and should not be confused with similar notation elsewhere in the paper.}
\[
L_\theta^n=C(n)+D_\theta(n), \qquad
D_\theta(n)=A(n)\star T_\theta(n)\star B_\theta(n),
\]
where 
\[
A(n):L^\infty(\bY)\to L^1(\bDelta), \quad
B_\theta(n):L^\infty(\bDelta)\to L^\infty(\bY), \quad
C(n):L^\infty(\bDelta)\to L^1(\bDelta),
\]
are given by
\[
(A(n)u)(y,\ell)=\begin{cases} u(y) & \ell=n \\ 0 & \text{else}
\end{cases},
\qquad
(C(n)u)(y,\ell)=\begin{cases} u(y,\ell-n) & \ell>n \\ 0 & \text{else}
\end{cases},
\]
and
\[
(B_\theta(n)u)(y)=\theta\sum_{a\in\alpha}\xi(y_a)1_{\{\Phi(a)>n\}}u(y_a,\Phi(a)-n).
\]
Here $\xi$ satisfies~\eqref{eq:GM}.
Hence
\[
\int_\bDelta|A(n)u|\,d\bar\mu_\Delta\le |u|_\infty \bar\mu_\Delta \{(y, n):y\in Y\}=\bPhi^{-1}|u|_\infty\mu_Y(\Phi>n),
\]
and
\[
|B_\theta(n)u|_\infty\ll |u|_\infty\sum_{a\in\alpha}\mu_Y(a)1_{\{\Phi(a)>n\}}
=|u|_\infty\,\mu_Y(\Phi>n).
\]
In other words,
\[
{|A(n)|}_{L^\infty(\bY)\mapsto L^1(\bDelta)}\le \mu_Y(\Phi>n),
\quad
{|B_\theta(n)|}_{L^\infty(\bDelta)\mapsto L^\infty(\bY)}\ll \mu_Y(\Phi>n).
\]
Combining this with the estimate for $T_\theta(n)$ we obtain 
$|D_\theta(n)1|_1\ll n^{-\beta'}$.
In particular,
\[
|1_\bZ D_\theta(n)1|_1 \ll n^{-\beta'}
\quad\text{and}\quad
|D_\theta(n)1_\bY|_1 \ll n^{-\beta'}.
\]

Finally, 
\[
|(C(n)u)(y,\ell)|=1_{\{\ell>n\}}|u(y,\ell-n)|\le 1_{\{\Phi(y)>n\}}|u|_\infty,
\]
so 
\[
|1_\bZ C(n)1|_1\le 
\bPhi^{-1}\int_\bY 1_{\{\Phi(y)>n\}}\sum_{\ell=0}^{\Phi(y)-1}1_{\bZ}(y,\ell)\,d\bar\mu_Y
\le \bPhi^{-1}\int_\bY 1_{\{\Phi>n\}}\sigma\,d\bar\mu_Y
=\bPhi^{-1}\sigma_{nd}
\]  
by~\eqref{eq:sigma}.
Also $C(n)1_\bY\equiv0$.
This completes the proof.
\end{proof}

\begin{cor} \label{cor:Ltheta}
There exists $C>0$ such that
\[
|\tv\circ f_\Delta^n-\tv_n|_1\le C|\tv|_\theta\sigma_{nd}
\quad\text{for all $\tv\in\cF_\theta(\hX)$, $n\ge1$.}
\]
\end{cor}

\begin{proof}
Note that $\theta^{\psi_n}$ is well-defined on $\bDelta$.  
By Proposition~\ref{prop:tv}(b,c), 
\[
\int_\Delta|\tv\circ f_\Delta^n-\tv_n|\,d\mu_\Delta
\ll |\tv|_\theta\int_\Delta \theta^{\psi_n}\,1_\hX\circ f_\Delta^n\,d\mu_\Delta
= |\tv|_\theta\int_\bDelta \theta^{\psi_n}\,1_\bZ\circ \bar f_\Delta^n\,d\bar\mu_\Delta.
\]
Hence the result follows from Lemma~\ref{lem:Ltheta}.
\end{proof}

\begin{pfof}{Lemma~\ref{lem:Delta2}(b)}
For $k\ge1$, let $a_k=\mu_\Delta(\hX)^{-1}\int_\Delta\tv_k\,d\mu_\Delta$.
Write
$\rho^*_{\tv,\tw}(n)=\int_\Delta \tv\circ f_\Delta^k\,\tw\circ f_\Delta^{k+n}\,d\mu_\Delta=I_1(k,n)+I_2(k,n)+I_3(k,n)+I_4(k,n)$,
where
\begin{align*}
I_1(k,n) & =\int_\Delta (\tv\circ f_\Delta^k-\tv_k)\,\tw\circ f_\Delta^{k+n}\,d\mu_\Delta, \\
I_2(k,n)  & =\int_\Delta \tv_k\,(\tw\circ f_\Delta^k-\tw_k)\circ f_\Delta^{n}\,d\mu_\Delta, \\ 
I_3(k,n)  &= \mu_\Delta(\hX)^{-1}\int_\Delta \tv_k\,d\mu_\Delta\,\int_\Delta 1_\hX\circ f^k\,\tw_k\circ f_\Delta^n\,d\mu_\Delta, \\ 
I_4(k,n)  & = \int_\Delta (\tv_k-a_k1_\hX\circ f_\Delta^k)\,\tw_k\circ f_\Delta^{n}\,d\mu_\Delta.
\nonumber
\end{align*}

By assumption, $\int_\Delta \tv\,d\mu_\Delta=0$.
Hence $\int_\Delta \tv_k\,d\mu_\Delta=
\int_\Delta (\tv_k-\tv\circ f_\Delta^k)\,d\mu_\Delta$.
By Corollary~\ref{cor:Ltheta},
\[
|I_1(k,n)|\ll |\tv|_\theta |\tw|_\infty  \sigma_{k d}, \quad
|I_2(k,n)|\ll |\tv|_\infty |\tw|_\theta \sigma_{k d}, \quad
|I_3(k,n)|\ll |\tv|_\theta |\tw|_\infty  \sigma_{k d}.
\]

Now,
\begin{equation} \label{eq:I4}
I_4(k,n)  = \int_\bDelta (\bar v_k-a_k1_\bZ\circ\bar f_\Delta^k)\,\bar w_k\circ \bar f_\Delta^{n}\,d\bar\mu_\Delta
 = \int_\bDelta u_k\,\bar w_k\circ \bar f_\Delta^{n-k}\,d\bar\mu_\Delta
\end{equation}
where 
\[
u_k=L^k(\bar v_k-a_k1_\bZ\circ\bar f_\Delta^k)
=L^k\bar v_k-a_k 1_\bZ.
\]
Note that $I_4$ is defined on $\bDelta$ and $u_k$ is supported in $\bZ$ (by Proposition~\ref{prop:tv}(b)) with
$\int_\bDelta u_k\,d\bar\mu_\Delta=0$.  Hence in the notation of Section~\ref{sec:Delta},
\(
I_4(k,n)=\rho^*_{u_k,\bar w_k}(n-k),
\)
and by Theorem~\ref{thm:Delta}(b),
\[
|I_4(k,n)|\ll 
\|u_k\|_\theta |\bar w_k|_\infty \gamma_{(n-k)d}.
\]
By Proposition~\ref{prop:tv}(a), $|u_k|_\infty\le 2|\tv|_\infty$ and
$|\bar w_k|_\infty\le |\tw|_\infty$.
The same argument as in the proof of Proposition~\ref{prop:hXbar} shows that
$1_\bZ(y,\ell)=1_\bZ(y',\ell)$ whenever $s(y,y')\ge1$ (so $y,y'$ lie in the same partition element).
Hence $|1_\bZ|_\theta\le2$.
By Proposition~\ref{prop:tv}(d), $|u_k|_\theta\le |L^k\bar v_k|_\theta+|a_k||1_\bZ|_\theta\ll \|\tv\|_\theta$.
Hence 
$|I_4(k,n)|\ll \|\tv\|_\theta |\tw|_\infty \gamma_{(n-k)d}$.

The combined estimates for $I_1,\,I_2,\,I_3,\,I_4$ give
$|\rho^*_{\tv,\tw}(n)|\ll \|\tv\|_\theta\|\tw\|_\theta (\sigma_{kd}+\gamma_{(n-k)d})$.
Taking $k=[n/2]$ yields the desired result.
\end{pfof}

\subsection{Completion of the proof of Theorem~\ref{thm:main2}}
\label{sec:proof2}

It remains to prove Lemma~\ref{lem:Delta2}(a).
There is the complication that approximation of observables $\tw$ supported on $\hX$  leads to one-sided observables $\tw_n$ that are not supported on $\hX$.  Hence Theorem~\ref{thm:Delta}(a) is not directly applicable.
The main new idea for dealing with this is the following:

\begin{lemma} \label{lem:A2}
There is a constant $C>0$ such that
\[
\Big|\int_\bY \sum_{\ell=0}^{\Phi(y)-1}1_{\{\ell> n\}} \bar w_k(y,\ell)\,d\bar\mu_Y\Big|
\le C\|\tw\|_\theta(\sigma_{kd}+n^{-\beta'}k^{\beta'}\zeta_{\beta'}(k) )
\]
for all $\tw\in\cF_\theta(\hX)$ with $\int_\Delta \tw \,d\mu_\Delta=0$, $n\ge k\ge1$.
\end{lemma}

\begin{proof}
Write $\int_\bY \sum_{\ell=0}^{\Phi(y)-1}1_{\{\ell> n\}} \bar w_k(y,\ell)\,d\bar\mu_Y =B_1+B_2$ where
\[
B_1  =\int_Y \sum_{\ell=0}^{\Phi-1}1_{\{\ell> n\}} (\tw_k-\tw\circ f_\Delta^k)(\cdot,\ell)\,d\mu_Y,  \qquad
B_2  =\int_Y
\sum_{\ell=0}^{\Phi-1}1_{\{\ell>n\}} \tw\circ f_\Delta^k(\cdot,\ell)\,d\mu_Y.
\]
By Corollary~\ref{cor:Ltheta},
\(
|B_1|\le \bPhi\int_\Delta |\tw_k-\tw\circ f_\Delta^k|\,d\mu_\Delta
\ll |\tw|_\theta\, \sigma_{kd}.
\)

Write $B_2=B_2'+B_2''$
where
\begin{align*}
B_2' & =\int_Y 1_{\{\Phi>n\}} \sum_{\ell=0}^{\Phi-k-1}1_{\{\ell>n\}}\tw\circ f_\Delta^k(\cdot,\ell)\,d\mu_Y,  \\
B_2''  & =\int_Y 1_{\{\Phi>n\}}\sum_{\ell=\Phi-k}^{\Phi-1}1_{\{\ell>n\}}\tw\circ f_\Delta^k(\cdot,\ell)\,d\mu_Y.
\end{align*}
Since $\tw=\tw 1_\hX$ and using~\eqref{eq:sigma},
\begin{align*}
|B_2'| & 
\le 
  |\tw|_\infty \int_\bY  \sum_{\ell=0}^{\Phi-1} 1_{\{\ell>k+n\}} 1_\bZ(\cdot,\ell)\,d\bar\mu_Y
  \le |\tw|_\infty \int_\bY 1_{\{\Phi>k\}} \sum_{\ell=0}^{\Phi-1}1_\bZ(\cdot,\ell)\,d\bar\mu_Y
 \\ & \le |\tw|_\infty \int_\bY 1_{\{\Phi>k\}} \sigma\,d\bar\mu_Y
= |\tw|_\infty\sigma_{kd}.
\end{align*}

Next,
\begin{align*}
B_2'' & = \int_Y 1_{\{\Phi>n\}} \sum_{\ell=\Phi-k}^{\Phi-1} 1_{\{\ell>n\}} \tw\circ f_\Delta^{\ell+k}(\cdot,0) \,d\mu_Y
  = \int_Y 1_{\{\Phi>n\}} \sum_{j=\Phi}^{\Phi+k-1} 1_{\{j>k+n\}} \tw\circ f_\Delta^{j}(\cdot,0) \,d\mu_Y
 \\ & = \int_Y 1_{\{\Phi>n\}} \sum_{j=0}^{k-1} 1_{\{j>k+n-\Phi\}} \tw\circ f_\Delta^{j}(F\cdot,0) \,d\mu_Y =E_1+E_2
\end{align*}
where
\begin{align*}
E_1 & =  \int_Y \sum_{j=0}^{k-1} 1_{\{\Phi>k+n-j\}}(\tw\circ f_\Delta^{j}-\tw_j)(F\cdot,0) \,d\mu_Y,  \\
E_2 & =
  \int_Y \sum_{j=0}^{k-1} 1_{\{\Phi>k+n-j\}}\tw_j(F\cdot,0) \,d\mu_Y.
\end{align*}
By Proposition~\ref{prop:tv}(c),
\[
|E_1|  \ll |\tw|_\theta \sum_{j=0}^{k-1}\int_Y1_{\{\Phi>n\}}\theta^{\psi_j\circ F}\,d\mu_Y
= |\tw|_\theta \sum_{j=0}^{k-1}\int_\bY \Psi_n\,\theta^{\psi_j}\,d\bar\mu_Y,
\]
where $\Psi_n=R1_{\{\Phi>n\}}$.
By Proposition~\ref{prop:GS} and Lemma~\ref{lem:Ltheta},
\[
\int_\bY \Psi_n\,\theta^{\psi_j}\,d\bar\mu_Y
 \le |\Psi_n|_\infty 
\int_\bDelta 1_{\bY} \theta^{\psi_j}\,d\bar\mu_\Delta
 \ll \mu_Y(\Phi>n) j^{-\beta'}.
\]
Hence $|E_1|\ll |\tw|_\theta \mu_Y(\Phi>n)$.

It remains to deal with $E_2$.  Now
\begin{align*}
 \int_Y1_{\{\Phi>k+n-j\}}\tw_j(F\cdot,0) \,d\mu_Y & =
 \int_\bY1_{\{\Phi>k+n-j\}}\bar w_j(\bF\cdot,0) \,d\bar\mu_Y 
=\int_\bY \Psi_{k+n-j} \bar w_j(\cdot,0) \,d\bar\mu_Y.
\end{align*}
Define $\bar u_n:\bDelta\to\R$, 
$u_n:\Delta\to\R$, 
\[
\bar u_n(y,\ell)=\begin{cases}
\Psi_n(y) & \ell=0
\\ 0 & \ell>0 \end{cases}, \qquad u_n=\bar u_n\circ\bar\pi.
\]
Then
\begin{align*}
\int_\bY \Psi_{k+n-j}\bar w_j(\cdot,0) \,d\bar\mu_Y & =
\bPhi\int_\bDelta \bar u_{k+n-j} \bar w_j \,d\bar\mu_\Delta  =
\bPhi\int_\Delta u_{k+n-j} \tw_j \,d\mu_\Delta
\\ & =\bPhi(G_1(j,k+n)+G_2(j,k+n))
\end{align*}
where
\[
G_1(j,n)=\int_\Delta u_{n-j} (\tw_j-\tw\circ f_\Delta^j) \,d\mu_\Delta, \qquad
G_2(j,n)=\int_\Delta u_{n-j} \tw\circ f_\Delta^j \,d\mu_\Delta.
\]

Applying Proposition~\ref{prop:GS} once more,
\[
|\Psi_n|_{L^\infty(\bY)} \ll \mu_Y(\Phi>n)
\quad\text{and}\quad
|\Psi_n(y)- \Psi_n(y')| \ll \mu_Y(\Phi>n) \theta^{s(y,y')}
\]
for all $y,y'\in \bY$, $n\ge1$.  Hence 
$|u_n|_\infty \ll \mu_Y(\Phi>n)$ and
$|u_n(y,\ell)-u_n(y',\ell)|\ll \mu_Y(\Phi>n) \theta^{s(y,y')}$
for all $y,y'\in Y$, $0\le \ell<\Phi(y)$.  That is, in the notation of Appendix~\ref{app:upper},
$\|u_n\|_\theta\ll \mu_Y(\Phi>n)$.

It follows that
\[
|G_1(j,n)|\le |u_{n-j}|_\infty|1_Y(\tw_j-\tw\circ f_\Delta^j)|_1
\ll |\tw|_\theta\,j^{-\beta'}(n-j)^{-\beta'} .
\]
Since $\tw$ has mean zero, we can apply Theorem~\ref{thm:upper} to obtain
\[
|G_2(j,n)| \ll \|u_{n-j}\|_\theta \|\tw\|_\theta\, j^{-(\beta'-1)}
\ll \|\tw\|_\theta\,j^{-(\beta'-1)}(n-j)^{-\beta'}.
\]
We conclude that
\[
|(G_1+G_2)(j,k+n)| 
\ll \|\tw\|_\theta\, j^{-(\beta'-1)}(k+n-j)^{-\beta'}
\ll \|\tw\|_\theta\, j^{-(\beta'-1)}n^{-\beta'}
\]
for all $j<k$, 
and hence
\[
|E_2(k,n)|\ll \|\tw\|_\theta\,n^{-\beta'}\sum_{j=0}^{k-1}j^{-(\beta'-1)}
\ll \|\tw\|_\theta \, n^{-\beta'}k^{\beta'}\zeta_{\beta'}(k) 
\]
completing the proof.
\end{proof}

\begin{pfof}{Lemma~\ref{lem:Delta2}(a)}
Let $k\ge1$ and write
$\rho^*_{1_\hX,\tw}(n)=I_2(k,n)+I_4(k,n)$,
where
\begin{align*}
I_2(k,n) & =\int_\Delta 1_\hX\,(\tw\circ f_\Delta^k-\tw_k)\circ f_\Delta^{n-k}\,d\mu_\Delta,   \\
 I_4(k,n) &  = \int_{\Delta} 1_\hX\,\tw_k\circ f_\Delta^{n-k}\,d\mu_\Delta
  = \int_{\bDelta} 1_\bZ\,\bar w_k\circ \bar f_\Delta^{n-k}\,d\bar\mu_\Delta.
\end{align*}
By Corollary~\ref{cor:Ltheta},
\(
|I_2(k,n)|\ll |\tw|_\theta\, \sigma_{k d}.
\)

Note that $I_4(k,n)$ is defined on $\bDelta$ and 
in the notation of Section~\ref{sec:Delta},
$I_4(k,n)=\rho^*_{1_\bZ,\bar w_k}(n-k)$.
Hence we can proceed almost as in the proof of Theorem~\ref{thm:Delta}(a), with $\tv$ and $\tw$ replaced by $1_\bZ$ and $\bar w_k$ respectively.  Let $m=n-k$.
Following Proposition~\ref{prop:conv}, we write
\begin{equation} \label{eq:I44}
I_4(k,n)=\bJ_{0,k}(m)+\bPhi^{-1}\int_\bY (T(m)\star R\bV(m))\star\bW_k(m)\,d\bar\mu_Y,
\end{equation}
where
\[
\bV(m)(y)=1_{\{\Phi(y)\ge m\}}1_\bZ(y,\Phi(y)-m),
\qquad
\bW_k(m)(y)= 1_{\{\Phi(y)>m\}}\bar w_k(y,m),
\]
and
\[
\bJ_{0,k}(m)  =\int_\bDelta 1_{\{m+\ell< \Phi(y)\}}1_\bZ(y,\ell)\bar w_k(y,m+\ell)\,d\bar\mu_\Delta.
\]
Also, define
\[
\bA_1(m)(y)=1_{\{\Phi(y)>m\}}\sum_{\ell=0}^{\Phi(y)-m-1}1_\bZ(y,\ell),
\qquad
\bA_{2,k}(m)(y)=\sum_{\ell=0}^{\Phi(y)-1}1_{\{m<\ell\}} \bar w_k(y,\ell).
\]

By Proposition~\ref{prop:tv}(a), $|\bar w_k|_\infty\le |\tw|_\infty$.
Proceeding exactly as in Section~\ref{sec:X}, we obtain the estimates
\begin{align*}
& |\bV(m)|_1\le \mu_Y(\Phi\ge m), \qquad
|\bW_k(m)|_1\le |\tw|_\infty\mu_Y(\Phi>m),
\\
& 
|\bJ_{0,k}(m)|\le \bar\Phi^{-1}|\tw|_\infty\,\sigma_{md},
\qquad 
|\bA_1(m)|_1 \le \sigma_{md}.
\end{align*}
Moreover,
$\bV(m)(y)=\bV(m)(y')$ 
for $y,y'\in a$, $a\in\alpha$.
Hence following the proof of Proposition~\ref{prop:Vn}, $R\bV(m)\in\cF_\theta(\bY)$ with $\|R\bV(m)\|_\theta\ll \mu_Y(\Phi>m)$.

By assumption, $\int_\Delta \tw\,d\mu_\Delta=0$.  Hence by Lemma~\ref{lem:A2}, $|P\bA_{2,k}(m)|\ll  \|\tw\|_\theta(\sigma_{kd}+ m^{-\beta'}k^{\beta'}\zeta_{\beta'}(k))$.

We have now estimated all the expressions arising in the 
proof of Theorem~\ref{thm:Delta}(a).
Continuing as in that proof, we obtain (cf.\ \eqref{eq:proof})
\[
|I_4(k,n)|\ll E_k(m)+\|\tw\|_\theta(\sigma_{md}+\zeta_{\beta'}(m))
\]
where $E_k(m)=b(m)\star PV(m)\star PW_k(m)$ and
\[
\hE_k(z)=\hat b(z)P\hV(1)P\hW_k(1)+(z-1)\hat b(z)\big\{
P\hV(z)P\hA_{2,k}(z)+P\hA_1(z)P\hW_k(1)\big\}.
\]
Recall that the Fourier coefficients of $(z-1)\hat b(z)$ are $O(m^{-\beta'})$.
Also,
\[
P\hW_k(1)=\int_\Delta \tw_k\,d\mu_\Delta=
\int_\Delta (\tw_k-\tw\circ f_\Delta^k)\,d\mu_\Delta 
\ll |\tw|_\theta\,\sigma_{kd} 
\]
by Corollary~\ref{cor:Ltheta}.
Hence
\begin{align*}
|E_k(m)| & \ll 
\|\tw\|_\theta\big(b(m)\sigma_{kd}+m^{-\beta'}\star\{m^{-\beta'}\star (\sigma_{kd}+
m^{-\beta'}k^{\beta'}\zeta_{\beta'}(k) ) +\sigma_{md} \sigma_{kd} \} \big)
\\ &
\ll \|\tw\|_\theta(\sigma_{kd}+m^{-\beta'}k^{\beta'}\zeta_{\beta'}(k) ).
\end{align*}
Hence
\[
|I_4(k,n)| \ll 
 \|\tw\|_\theta(\sigma_{kd}+\sigma_{(n-k)d}+(n-k)^{-\beta'}k^{\beta'}\zeta_{\beta'}(k) +\zeta_{\beta'}(n-k) ).
\]
Combining this with the estimate for $I_2(k,n)$ and taking $k=[n/2]$ yields the desired result.
\end{pfof}

\section{Billiard examples}
\label{sec:billiard}

In this section, we provide details and proofs for the examples considered in Section~\ref{sec:introbill}.
For background material on billiards, we refer to~\cite{ChernovMarkarian}.
The billiard domain, denoted by $Q$, is a compact connected subset of $\R^2$ or $\T^2$ with piecewise smooth boundary and the billiard flow is defined on $Q\times S^1$.   Fix a point $q\in Q$ and a unit vector $v\in S^1$.
Then $q$ moves in straight lines with unit speed in direction $v$ until reflecting (angle of reflection equalling the angle of incidence) off the boundary $\partial Q$.    This defines a volume-preserving flow.
A natural Poincar\'e section is given by $M=\partial Q\times[-\pi/2,\pi/2]$ corresponding to collisions with $\partial Q$ (with outgoing velocities in $[-\pi/2,\pi/2]$).  The Poincar\'e map $f:M\to M$ is called the collision map or the billiard map.  It preserves a probability measure $\mu$, equivalent to Lebesgue, called Liouville measure.

Part of the framework in~\cite{ChernovZhang05,Markarian04}
is that the billiard map $f:M\to M$ has a 
(two-sided) Chernov-Markarian-Zhang structure as defined in Section~\ref{sec:main2}.
In particular, $f$ has a
suitably chosen first return map $f_X=f^h:X\to X$ modelled by a Young tower $\DE=Y^\varphi$ with exponential tails.  
Roughly speaking, $X$ is chosen to be a subset of phase space bounded away from the regions where hyperbolicity is expected to break down, e.g.\  for billiards with cusps, $X$ excludes a neighborhood of each cusp.
Since the specific choice of $X$ involves notation which is not required for understanding the results, we mainly point the reader to the original references for the precise definitions. (An exception is Example~\ref{ex:semi} below, where no extra notation is needed.)

\begin{example}[Bunimovich stadia~\cite{Bunimovich79}] 
\label{ex:stadia}
These are convex billiard domains $Q\subset\R^2$ where $\partial Q$
is a simple closed curve consisting of two parallel line segments and two semicircles.
By~\cite{Markarian04}, 
the billiard map $f:M\to M$ falls within the Chernov-Markarian-Zhang framework
with $\mu_X(h>n)=O(n^{-2})$.  
By~\cite[Theorem~1.1]{ChernovZhang08}, 
 $\mu_Y(\varphi>n)=O(n^{-2})$ and hence
$\rho_{v,w}(n)=O(n^{-1})$ for dynamically H\"older observables.

Here, we improve the estimate on $\mu_Y(\varphi>n)$ and use this to obtain lower bounds on decay of correlations.

\begin{prop} \label{prop:Bun}
For Bunimovich stadia, there exists $c>0$ such that $\bvarphi^{-1}\mu_Y(\varphi>n)\sim cn^{-2}$
and 
$\rho_{v,w}(n)\sim c n^{-1}\int_M v\,d\mu\int_M w\,d\mu$ for all $v,\,w\in\cH(X)$ with nonzero mean.

In addition $\rho_{v,w}(n)=O(n^{-2}\log n)$ for all $v\in \cH(X)$ with $\int v\,d\mu=0$ and all $w\in\cH(X)$.
\end{prop}

\begin{proof}
In the proof of~\cite[Theorem~1.1]{BalintGouezel06}
(see in particular~\cite[page 504, line 11]{BalintGouezel06})
it is shown for $h:X\to\Z^+$ (denoted there by $\varphi_+$) that
$(n\log n)^{-1/2}(\sum_{j=0}^{n-1}h\circ f_X^j-n\int_X h\,d\mu_X)$ converges to a nondegenerate normal distribution.
Hence the first statement follows from
Corollary~\ref{cor:tails} and the second statement from
Theorem~\ref{thm:main2}(a).

Finally, $\gamma_n=O(n^{-2}\log n)$ by Proposition~\ref{prop:int}, so the final statement follows from
Theorem~\ref{thm:main2}(b).
\end{proof}
\end{example}

\begin{example}[Semidispersing billiards]  
\label{ex:semi}
The billiard domain is given by
$Q=R\setminus\bigcup S_k$ where $R$ is a rectangle and there are finitely many disjoint convex scatterers $S_k\subset R$ with $C^3$ boundaries of nonvanishing curvature.

By~\cite[Theorem~1]{ChernovZhang05}, the billiard map $f:M\to M$  falls within the Chernov-Markarian-Zhang framework
with 
$X=\bigcup\partial S_k\times[-\pi/2,\pi/2]$ and
$\mu_X(h>n)=O(n^{-2})$. 
By~\cite[Theorem~1.1]{ChernovZhang08}, 
$\mu_Y(\varphi>n)=O(n^{-2})$ and hence 
$\rho_{v,w}(n)=O(n^{-1})$ for dynamically H\"older observables.

\begin{prop} \label{prop:semi}
The conclusions of Proposition~\ref{prop:Bun} hold
for semidispersing billiards.
\end{prop}

\begin{proof}
Note that $h$ is precisely the free flight time considered in~\cite{Bleher92,SzaszVarju07}.
By~\cite[Theorem~1]{SzaszVarju07}, $(n\log n)^{-1/2}(\sum_{j=0}^{n-1}h\circ f_X^j-n\int_X h\,d\mu_X)$ converges to a nondegenerate normal distribution.
Now proceed as in the proof of Proposition~\ref{prop:Bun}.
\end{proof}

\end{example}

\begin{example}[Billiards with cusps]
These are billiard domains $Q\subset\R^2$ where $\partial Q$ is a simple closed curve consisting of finitely many convex inwards $C^3$ curves with nonvanishing curvature such that the interior angles at corner points are zero.

By~\cite[Theorem~1.1]{ChernovMarkarian07}, the billiard map $f:M\to M$ falls within the Chernov-Markarian-Zhang framework
with $\mu_X(h>n)=O(n^{-2})$.
By~\cite[Theorem~1.1]{ChernovZhang08}, 
$\mu_Y(\varphi>n)=O(n^{-2})$ and hence 
$\rho_{v,w}(n)=O(n^{-1})$ for dynamically H\"older observables.

\begin{prop} \label{prop:cusp}
The conclusions of Proposition~\ref{prop:Bun} hold
for billiards with cusps.
\end{prop}

\begin{proof}
By~\cite[Theorem~4 and eq.~(2.5)]{BalintChernovDolgopyat11}, $(n\log n)^{-1/2}(\sum_{j=0}^{n-1}h\circ f_X^j-n\int_X h\,d\mu_X)$ converges to a nondegenerate normal distribution.
Now proceed as in the proof of Proposition~\ref{prop:Bun}.
\end{proof}

\end{example}

\begin{example}[Billiards with cusps at flat points~\cite{Zhang17b}]
These are billiard domains $Q\subset\R^2$ where $\partial Q$ is a simple closed curve consisting of finitely many convex inwards $C^3$ curves such that the interior angles at one of the corner points is zero. Moreover the curves have nonvanishing curvature except at this corner point where $\partial Q$ has the form $\pm x^b$ for some $b>2$.

By~\cite{Zhang17b}, 
the billiard map $f:M\to M$ 
falls within the Chernov-Markarian-Zhang framework
with $\mu_X(h>n)=O(n^{-\beta})$ where $\beta=b/(b-1)\in(1,2)$.
Moreover, by~\cite{Zhang17b}, 
$\mu_Y(\varphi>n)=O(n^{-\beta})$ and hence 
$\rho_{v,w}(n)=O(n^{-(\beta-1)})$ for dynamically H\"older observables.

\begin{prop} \label{prop:JZ}
For billiards with cusps at flat points, there exists $c>0$ such that $\mu_Y(\varphi>n)\sim cn^{-\beta}$
and 
$\rho_{v,w}(n)\sim \bvarphi^{-1}(\beta-1)^{-1}c n^{-(\beta-1)}\int_M v\,d\mu\int_M w\,d\mu$ for all $v,\,w\in\cH(X)$ with nonzero mean.

In addition $\rho_{v,w}(n)=O(n^{-\beta}\log n)$ for all $v\in \cH(X)$ with $\int v\,d\mu=0$ and all $w\in\cH(X)$.
\end{prop}

\begin{proof}
By~\cite[Theorem~3.1]{JungZhangapp}, $n^{-1/\beta}(\sum_{j=0}^{n-1}h\circ f_X^j-n\int_X h\,d\mu_X)$ converges to a nondegenerate $\beta$-stable law.
Hence the first statement follows from
Corollary~\ref{cor:tails} and the second statement from
Theorem~\ref{thm:main2}(a).

Finally, $\gamma_n=O(n^{-\beta}\log n)$ by Proposition~\ref{prop:int}, so the final statement follows from
Theorem~\ref{thm:main2}(b).
\end{proof}

\end{example}

\paragraph{Explicit formulas for asymptotic constants}
In all the billiard examples considered above, the constant $c$ 
can be made completely explicit.  
For brevity, we restrict to the case of 
the stadium in Example~\ref{ex:stadia}.
Let $\ell$ be the length of the parallel line segments in $\partial Q$.
The argument in~\cite[page 504, line 11]{BalintGouezel06} shows that
\[
(c_0n\log n)^{-1/2}\Big(\sum_{j=0}^{n-1}h\circ f_X^j-n\int_X h\,d\mu_X\Big)\to_d N(0,1)
\quad\text{where $c_0=\frac{4+3\log 3}{4-3\log 3}\,\frac{\ell^2}{8}$}.
\]
Since $\varphi=h_\sigma$, it follows from Lemma~\ref{lem:CLT}(a) that
\[
B_n^{-1}\Big(\sum_{j=0}^{n-1}\varphi\circ F^j-n\int_Y \varphi\,d\mu_Y\Big)\to_d  N(0,1)
\]
where $B_n=(c_1n\log n)^{1/2}$ and $c_1=\bar\sigma c_0$.
Define $L(x)=2c_1\log x$.
Then $nL(B_n)\sim B_n^2$ and $L(x)=2\int_1^x c_1 u^{-1}du$.
Applying~\cite[Theorem~1.5]{Gouezel10b}, we obtain $\mu_Y(\varphi>n)\sim c_1 n^{-2}$.

Next, $\bar\varphi=\bar\sigma\int_X h\,d\mu_X=\bar\sigma/\mu(X)$
and $\mu(X)=2/(\pi+\ell)$ by~\cite[Eq.~(6)]{BalintGouezel06}.
Hence 
\[
\bvarphi^{-1}\mu_Y(\varphi>n)\sim \frac{2c_0}{\pi+\ell}\,n^{-2}=
\frac{4+3\log 3}{4-3\log 3}\,\frac{\ell^2}{4(\pi+\ell)}\,n^{-2}.
\]
 By Theorem~\ref{thm:main2}(a),
\[
\rho_{v,w}(n)\sim 
 \frac{4+3\log 3}{4-3\log 3}\,\frac{\ell^2}{4(\pi+\ell)}\,n^{-1}  \int_M v\,d\mu\int_M w\,d\mu
\]
for $v,w\in\cH(X)$ with nonzero mean.

\begin{example}[Bunimovich flowers~\cite{Bunimovich73}]
These are billiard domains $Q\subset\R^2$ where $\partial Q$ is a simple closed piecewise $C^3$ curve consisting of at least one arc with nonvanishing curvature that is convex inwards and at least one convex outwards circular arc that is strictly smaller than a semicircle.  All corner points have nonzero angle, and each convex outwards arc continues to a circle contained in $Q$.  (The conditions can be further relaxed to allow line segments in $\partial Q$, see~\cite{ChernovZhang05}.)

By~\cite[Theorem~2]{ChernovZhang05}, the billiard map
$f:M\to M$ falls within the Chernov-Markarian-Zhang framework
with $\mu_X(h>n)=O(n^{-3})$.   Hence $\mu_Y(\varphi>n)=O((\log n)^3n^{-3})$ leading as in~\cite{Young99} to the upper bound on decay of correlations
$\rho_{v,w}(n)=O((\log n)^3n^{-2})$ for dynamically H\"older observables.

It is also easily verified that $\mu_X(h>n)\approx n^{-3}$.  Only the more delicate upper bound is explicit in~\cite{ChernovZhang05}, but  the lower bound is much simpler.  It suffices to estimate the contribution from sliding along a single convex outwards circular arc $S\subset\partial Q$.
Let $(r,\phi)$ denote coordinates on $S\times[-\pi/2,\pi/2]$ where $r\in[0,r_0]$ is arclength along $S$.
Then $\mu_X$ is given by $d\mu_X=\cos\phi\,d\phi\,dr$.
The set $X$ is chosen to exclude points that make at least two successive collisions with $S$.  Hence $\{h>n\}$ includes all points $(r,\phi)$ with $r$ close to the beginning of $S$ and $\phi$ close to $\pi/2$.
Since $S$ is circular, the angles at successive collisions remain close to this initial value of $\phi$ so it is clear that $\{h>n\}$ contains a set of the form
\[
E_n=\textstyle \{(r,\phi): 0\le r\le \frac{a}{n},\; \frac{\pi}{2}-\frac{b}{n} \le\phi\le\frac{\pi}{2}\},
\]
where $a$ and $b$ are constants independent of $n$.
Hence 
\[
\textstyle \mu_X(h>n)\ge \int_0^{r_0}\int_{-\pi/2}^{\pi/2}1_{E_n}\cos\phi\,d\phi\,dr
\sim\frac12 ab^2 n^{-3}.
\]

By Proposition~\ref{prop:rough}, it follows that $\mu_Y(\varphi>n)\gg (\log n)^{-1}n^{-3}$.  By Theorem~\ref{thm:main2}, $\rho_{v,w}(n)\gg (\log n)^{-1}n^{-2}\int_M v\,d\mu\int_M w\,d\mu$ for all $v,\,w\in \cH(X)$, $n\ge1$.

\end{example}

\begin{example}[Dispersing billiards with vanishing curvature~\cite{ChernovZhang05b}]
These are planar periodic dispersing billiards $Q=\T^2\setminus\bigcup S_k$ where there are finitely many disjoint strictly convex scatterers $S_k$ with $C^3$ boundaries of nonvanishing curvature, except that the curvature vanishes at two points.  Moreover, there is a periodic orbit that runs between these two points and the boundary nearby has the form $\pm(1+|x|^b)$ for some $b>2$.  
By~\cite[Theorem~1]{ChernovZhang05b}, the billiard map $f:M\to M$ falls within the Chernov-Markarian-Zhang framework
with $\mu_X(h>n)=O(n^{-\beta})$ where $\beta-1=(b+2)/(b-2)\in (1,\infty)$. Hence $\mu_Y(\varphi>n)=O((\log n)^\beta n^{-\beta})$ leading as in~\cite{Young99} to the upper bound on decay of correlations
$\rho_{v,w}(n)=O((\log n)^\beta n^{-(\beta-1)})$ for dynamically H\"older observables.

Moreover $\mu_X(h>n)\gg cn^{-\beta}$ by~\cite[Proposition~2]{ChernovZhang05b}.
By Proposition~\ref{prop:rough}, it follows that $\mu_Y(\varphi>n)\gg (\log n)^{-1}n^{-\beta}$.  By Theorem~\ref{thm:main2}, $\rho_{v,w}(n)\gg (\log n)^{-1}n^{(\beta-1)}\int_M v\,d\mu\int_M w\,d\mu$ for all $v,\,w\in \cH(X)$, $n\ge1$.
\end{example}

\appendix

\section{Formula for the correlation function}
\label{app:conv}

In this appendix, we prove Proposition~\ref{prop:conv}.
One method would be to check equality of coefficients directly, but we choose to
convert all sequences into Fourier series.
Recall that $V(n)(y)=1_{\{\Phi\ge n\}}\tv(y,\Phi(y)-n)$,
$W(n)(y)= 1_{\{\Phi>n\}}\tw(y,n)$, with Fourier series
\[
\hV(z)(y)=\sum_{\ell=0}^{\Phi(y)-1} z^{\Phi(y)-\ell}\tv(y,\ell), \qquad
\hW(z)(y)=\sum_{\ell=0}^{\Phi(y)-1} z^\ell\tw(y,\ell).
\]
Define $\Phi_k=\sum_{j=0}^{k-1}\Phi\circ F^j$.
Arguing as in~\cite[Section~6.2]{M18} (with discrete time instead of continuous time), 
write
\begin{align*}
\rho^*_{\tv,\tw}(n) & 
 =\int_\Delta 1_{\{n+\ell<\Phi(y)\}}\tv(y,\ell)\,\tw\circ f_\Delta^n(y,\ell)\,d\mu_\Delta
\\ & \qquad \qquad +\sum_{k=1}^\infty\int_\Delta 1_{\{\Phi_k(y)\le n+\ell<\Phi_{k+1}(y)\}}\tv(y,\ell)\,\tw\circ f_\Delta^n(y,\ell)\,d\mu_\Delta
= \sum_{k=0}^\infty J_k(n),
\end{align*}
where $J_0(n)$ was defined in Section~\ref{sec:X} and
\[
J_k(n)=
\int_\Delta 1_{\{\Phi_k(y)\le \ell+n<\Phi_{k+1}(y)\}}\tv(y,\ell)\tw(F^k y,\ell+n-\Phi_k(y))\,d\mu_\Delta,\;  k\ge1.
\]
For $k\ge1$,
\begin{align*}
\hJ_k(z) & =\sum_{n=0}^\infty z^nJ_k(n)
 =\bPhi^{-1}\int_Y\sum_{\ell=0}^{\Phi(y)-1}\sum_{n=\Phi_k(y)-\ell}^{\Phi_{k+1}(y)-\ell-1}z^n\tv(y,\ell)\tw(F^k y,\ell+n-\Phi_k(y))\,d\mu_Y.
\end{align*}
Making the substitution $\ell'=\ell+n-\Phi_k(y)$,
\begin{align*}
\hJ_k(z) & =\bPhi^{-1}\int_Y\Bigl(\sum_{\ell=0}^{\Phi(y)-1}z^{-\ell}\tv(y,\ell)\Bigr)
\Bigl(\sum_{\ell'=0}^{\Phi(F^k y)-1}z^{\ell'}\tw(F^k y,\ell')\Bigr)z^{\Phi_k(y)}\,d\mu_Y
\\ & = \bPhi^{-1}\int_Y z^{\Phi_k}v_z \,\hW(z)\circ F^k\,d\mu_Y,
\end{align*}
where $v_z(y)=\sum_{\ell=0}^{\Phi(y)-1}z^{-\ell}\tv(y,\ell)$.
Note that $z^{\Phi}v_z=\hV(z)$.  Hence
\begin{align*}
R^k(z^{\Phi_k}v_z)=R^{k-1}R(z^{\Phi_{k-1}\circ F}z^{\Phi}v_z)
=R^{k-1}(z^{\Phi_{k-1}}R(z^{\Phi}v_z))
=\hR(z)^{k-1}R\hV(z),
\end{align*}
and so
\begin{align*}
\hJ_k(z)
& = \bPhi^{-1}\int_Y R^k(z^{\Phi_k}v_z) \cdot \hW(z)\,d\mu_Y
= \bPhi^{-1}\int_Y \hR(z)^{k-1}R\hV(z) \cdot \hW(z)\,d\mu_Y.
\end{align*}
Hence,
\begin{align*}
\textstyle \hat\rho^*_{\tv,\tw}(z)
& = \sum_{k=0}^\infty \hJ_k(z) =\hJ_0(z)+\bPhi^{-1}\sum_{k=1}^\infty \int_Y \hR(z)^{k-1}R\hV(z)\cdot  \hW(z)\,d\mu_Y
\\ & =\hJ_0(z)+\bPhi^{-1}\int_Y \hT(z)R\hV(z)\cdot  \hW(z)\,d\mu_Y.
\end{align*}

This completes the proof of Proposition~\ref{prop:conv}.

\section{Upper bounds on Young towers}
\label{app:upper}

In this appendix, we recall a standard result giving an upper bound on decay of correlations for two-sided Young towers.  We could not find the result stated in the form we require in the literature; hence we provide the details.

Let $f_\Delta:\Delta\to\Delta$ be a mixing two-sided Young tower with
$\mu_\Delta(\Phi>n)=O(n^{-\beta'})$ where $\beta'>1$.
As in Section~\ref{sec:prel2}, $\Delta=Y^\Phi$ is a tower over a two-sided Gibbs-Markov map defined on a bounded metric space $(Y,d)$.
Fix $\theta\in(0,1)$ and define $\cF_\theta(\Delta)$ to be the space of observables $\tv:\Delta\to\R$ with $\|\tv\|_\theta<\infty$ where $\|\tv\|_\theta$ is defined as in~\eqref{eq:norm}. 
Note that if $v:M\to\R$ lies in $\cH(M)$, then $\tv=v\circ\pi_M\in\cF_\theta(\Delta)$
and $\|\tv\|_\theta={\|v\|}_\cH$.
However, here we consider observables on $\Delta$ that need not be lifts of observables on~$M$ and the underlying metric space $(M,d)$ plays no role.

\begin{thm} \label{thm:upper}
There exists a constant $C>0$ such that
\[
\Big| \int_\Delta \tv\,\tw\circ f_\Delta^n\,d\mu_\Delta
-\int_\Delta \tv\,d\mu_\Delta
\int_\Delta \tw\,d\mu_\Delta\Big|
\le C\|\tv\|_\theta \|\tw\|_\theta \, n^{-(\beta'-1)}
\]
for all $\tv,\tw\in\cF_\theta(\Delta)$, $n\ge1$.
\end{thm}

First, we mention some prerequisites.
Define $\cF_\theta(\bDelta)$ to consist of observables
$\bar v:\bDelta\to\R$ with $\|\bar v\|_\theta<\infty$ where
$\|\bar v\|_\theta$ is defined as in~\eqref{eq:norm1}.
Note that if $\bar v\in \cF_\theta(\bDelta)$, then
$\tv=\bar v\circ\bar\pi\in\cF_\theta(\Delta)$ and
$\|\tv\|_\theta =\|\bar v\|_\theta$.

For $\tv\in\cF_\theta(\Delta)$, define $\tv_n:\Delta\to\R$ and $\bar v_n:\bDelta\to\R$ as in Section~\ref{sec:approx}.

\begin{prop} \label{prop:prel}
There exists $C>0$ such that 
for all $\tv\in\cF_\theta(\theta)$, $n\ge1$,
\begin{itemize}
\item[(a)] $|\tv\circ f_\Delta^n-\tv_n|_1\le C|\tv|_\theta\,n^{-(\beta'-1)}$,
\item[(b)] $\|L^n\bar v_n\|_\theta \le C\|\tv\|_\theta$.
\end{itemize}
\end{prop}

\begin{proof}
(a) Arguing as in the proof of Proposition~\ref{prop:tv},
\(
|\tv\circ f_\Delta^n-\tv_n|
\ll |\tv|_\theta\, \theta^{\psi_n}.
\)
Proceeding as in the proof of Lemma~\ref{lem:Ltheta},
we can write
\[
\int_\bDelta \theta^{\psi_n}\,d\bar\mu_\Delta=
\int_\bDelta L^n\theta^{\psi_n}\,d\bar\mu_\Delta=
\int_\bDelta L_\theta^n1\,d\bar\mu_\Delta
\]
and $L_\theta^n=C(n)+D(n)$ where $|D(n)1|_1\ll n^{-\beta'}$.
Also, $|C(n)1|_1\le\bPhi^{-1}\int_\bY 1_{\{\Phi>n\}}\Phi\,d\bar\mu_Y
\ll n^{-(\beta'-1)}$.
Hence
\[
|\tv\circ f_\Delta^n-\tv_n|_1 \ll |\tv|_\theta
|\theta^{\psi_n}|_1 \ll |\tv|_\theta\,n^{-(\beta'-1)}.
\]

(b) The same calculations as in Proposition~\ref{prop:tv}(d), show that
$|L^n\bar v_n|_\infty\le |\tv|_\infty$ and
$|(L^n\bar v_n)(\bar p_1)-(L^n\bar v_n)(\bar p_2)|
\ll \|\tv\|_\theta\, \theta^{s(\bar p_1,\bar p_2)}$
for $\bar p_1,\bar p_2\in\bDelta$, $n\ge1$.
\end{proof}

\begin{pfof}{Theorem~\ref{thm:upper}}
First, recall that there is a constant $C>0$ such that
\begin{equation} \label{eq:oneside}
\Big| \int_\bDelta \bar v\,\bar w\circ \bar f_\Delta^n\,d\bar\mu_\Delta
-\int_\bDelta \bar v\,d\bar\mu_\Delta
\int_\bDelta \bar w\,d\bar\mu_\Delta\Big|
\le C\|\bar v\|_\theta |\bar w|_\infty \, n^{-(\beta'-1)}
\end{equation}
for all $\bar v\in\cF_\theta(\bDelta)$, $\bar w\in L^\infty(\bDelta)$, $n\ge1$.
(This follows from~\cite[Theorem~3]{Young99}, the specific dependence on
$\|\bar v\|_\theta$, $|\bar w|_\infty$ being a standard consequence 
of the uniform boundedness principle.  Alternatively, see~\cite[Section~2.2]{KKM19}.)

Suppose without loss of generality that $\int_\Delta \tw\,d\mu_\Delta=0$.
Write
\[
\int_\Delta \tv\,\tw\circ f_\Delta^{2n}\,d\mu_\Delta
 = \int_\Delta \tv\circ f^n\,\tw\circ f_\Delta^{3n}\,d\mu_\Delta
=I_1(n)+I_2(n)+I_3(n)
\]
where
\begin{align*}
I_1(n) & =  \int_\Delta (\tv\circ f_\Delta^n-\tv_n)\,\tw\circ f_\Delta^{3n}\,d\mu_\Delta, \quad
I_2(n)  =  \int_\Delta \tv_n\,(\tw\circ f_\Delta^n-\tw_n)\circ f_\Delta^{2n}\,d\mu_\Delta
\\  
I_3(n) & = \int_\Delta \tv_n\,\tw_n\circ f_\Delta^{2n} \,d\mu_\Delta
= \int_\bDelta \bar v_n\,\bar w_n\circ \bar f_\Delta^{2n} \,d\bar\mu_\Delta
= \int_\bDelta L^n\bar v_n\,\bar w_n\circ \bar f_\Delta^n \,d\bar\mu_\Delta.
\end{align*}
By Proposition~\ref{prop:prel}(a),
\[
|I_1(n)| \ll |\tv|_\theta |\tw|_\infty\,n^{-(\beta'-1)}, \qquad
|I_2(n)| \ll |\tv|_\infty |\tw|_\theta \,n^{-(\beta'-1)}.
\]
By~\eqref{eq:oneside},
\[
|I_3(n)| \ll \Big|\int_\bDelta L^n\bar v_n\,d\bar\mu_\Delta\Big|\,
\Big|\int_\bDelta \bar w_n\,d\bar\mu_\Delta\Big|+\|L^n\bar v_n\|_\theta|\bar w_n|_\infty\,n^{-(\beta'-1)}.
\]
Since $\tw$ has mean zero, it follows from Proposition~\ref{prop:prel}(a) that
$|\int_\bDelta \bar w_n\,d\bar\mu_\Delta|=
|\int_\Delta (\tw_n-\tw\circ f_\Delta^n)\,d\mu_\Delta|\ll |\tw|_\theta\,n^{-(\beta'-1)}$
Hence using Proposition~\ref{prop:prel}(b),
\[
|I_3(n)|\ll \|\tv\|_\theta\|\tw\|_\theta\,n^{-(\beta'-1)},
\]
and the result follows.
\end{pfof}

 \paragraph{Acknowledgements}
HB gratefully acknowledges the support of FWF grant P31950-N45 during
the last stages of writing this paper.
 The research of IM was supported in part by the European Advanced Grant {\em StochExtHomog} (ERC AdG 320977).
We are grateful to the referee for reading the paper carefully and making several very helpful suggestions, to Alexey Korepanov for helpful discussions and to Peyman Eslami for pointing out a problem in the previous version.

\end{document}